\include{birkart.sty}

\documentclass{amsart}[12pt]
\usepackage{amsmath}
\usepackage{amssymb}
\usepackage{amsfonts}
\usepackage{graphicx}
\usepackage{texdraw}
\usepackage{graphpap}
\usepackage{wasysym}
\usepackage[OT2,T1]{fontenc}

\input txdtools

\newtheorem{mth}{Theorem}

\newtheorem{thm}{Theorem}[section]
\newtheorem{cor}[thm]{Corollary}
\newtheorem{lem}[thm]{Lemma}
\newtheorem*{propA1}{Proposition A.1}

\newtheorem{prop}[thm]{Proposition}

\theoremstyle{definition}

\theoremstyle{remark}
\newtheorem{rem}[thm]{Remark}

\newtheorem*{ex}{Example}

\numberwithin{equation}{section}

\newcommand{\1}{{\mathbf{1}}}
\newcommand{\calA}{{\mathcal A}}
\newcommand{\bo}{{\mathcal L}}

\newcommand{\calH}{{\mathcal H}}
\newcommand{\calI}{{\mathcal I}}
\newcommand{\calJ}{{\mathcal J}}

\newcommand{\calK}{{\mathcal K}}
\newcommand{\calP}{{\mathcal P}}
\newcommand{\calM}{{\mathcal M}}
\newcommand{\calV}{{\mathcal V}}
\newcommand{\calF}{{\mathcal F}}
\newcommand{\calS}{{\mathcal S}}
\newcommand{\calW}{{\mathcal W}}
\newcommand{\typeH}{\mathbf{H}}
\newcommand{\Kpr}{$K$-property\,}

\newcommand{\R}{{\mathbf R}}

\newcommand{\C}{{\mathbf C}}
\newcommand{\im}{\operatorname{Im}}

\newcommand{\vro}{{\varrho}}
\newcommand{\fii}{{\varphi}}
\newcommand{\kfun}{{k}}
\newcommand{\esssup}{\operatorname{\text{\rm ess\,sup}\,}}
\newcommand{\essinf}{\operatorname{\text{\rm ess\,inf}\,}}
\newcommand{\diag}{\operatorname{diag}}
\newcommand{\re}{\operatorname{Re}}

\newcommand{\spa}{\operatorname{span}}

\newcommand{\eps}{{\varepsilon}}
\newcommand{\QI}{\operatorname{QI}}
\newcommand{\QM}{\operatorname{QM}}

\begin{document}

\title{Interpolation between Hilbert spaces}

\author{Yacin Ameur}

\address{Department of Mathematics\\Faculty of Science\\Lund University\\P.O. Box 118\\221 00 Lund\\Sweden}
\email{Yacin.Ameur@math.lu.se}

\begin{abstract}
This note comprises a synthesis of certain results in the theory of exact interpolation
between Hilbert spaces. In particular, we examine various characterizations of
interpolation spaces and their relations to a number of results in operator-theory and in function-theory.


\end{abstract}





\maketitle


\section{Interpolation theoretic notions}

\subsection{Interpolation norms} \label{internorm} When $X$, $Y$ are normed spaces, we use the symbol $\bo(X;Y)$ to denote
the totality of bounded linear maps $T:X\to Y$ with the operator norm
$$\left\|\,T\,\right\|_{\,\bo(X;Y)}=\sup\left\{\left\|\,Tx\,\right\|_Y;\, \left\|\,x\,\right\|_X\le 1\right\}.$$
When $X=Y$ we simply write $\bo(X)$.

Consider a pair of Hilbert spaces $\overline{\calH}=\left(\calH_0,\calH_1\right)$ which is \textit{regular} in
the sense that $\calH_0\cap \calH_1$ is dense in $\calH_0$ as well as in
$\calH_1$. We assume that the pair is \textit{compatible}, i.e., both $\calH_i$ are embedded in some common Hausdorff topological vector space
$\calM$.

We define the \textit{$K$-functional} (\footnote{More precisely, this is the \textit{quadratic version} of the classical Peetre $K$-functional.}) for the couple $\overline{\calH}$ by
$$K\left(t,x\right)=K\left(t,x\,;\,\overline{\calH}\,\right)=\inf_{x=x_0+x_1}
\{\,\left\|\,x_0\,\right\|_0^{\,2}+t\left\|\,x_1\,
\right\|_1^{\,2}\,\},\qquad t>0,\, x\in\calM.$$
The \textit{sum} of the spaces $\calH_0$ and $\calH_1$ is defined to be the space  consisting of all $x\in\calM$ such that the quantity
$\left\|\,x\,\right\|_\Sigma^{\,2}:=K\left(1,x\right)$ is finite; we denote this space by the symbols
$$\Sigma=\Sigma(\,\overline{\calH}\,)=\calH_0+\calH_1.$$
We shall soon see that $\Sigma$ is a Hilbert
space (see Lemma \ref{kcalc}).
The \textit{intersection} $$\Delta=\Delta(\,\overline{\calH}\,)=\calH_0\cap\calH_1$$
is a Hilbert space under the norm $\left\|\,x\,\right\|_\Delta^{\,2}:=\left\|\,x\,\right\|_0^{\,2}+
\left\|\,x\,\right\|_1^{\, 2}$.

A map $T:\Sigma(\,\overline{\calH}\,)\to\Sigma(\,\overline{\calK}\,)$ is called a \textit{couple map} from
$\overline{\calH}$ to $\overline{\calK}$ if the
restriction of $T$ to $\calH_i$ maps $\calH_i$ boundedly into $\calK_i$ for $i=0,1$.
We use the notations $T\in \bo(\,\overline{\calH}\,;\,\overline{\calK}\,)$ or
$T:\overline{\calH}\to\overline{\calK}$ to denote that $T$ is a couple map. It is easy to check
 that $\bo(\,\overline{\calH}\,;\,\overline{\calK}\,)$
is a Banach space,  when equipped with the norm
\begin{equation}\label{equi}\left\|\, T\,\right\|_{\,\bo(\,\overline{\calH}\,;\,\overline{\calK}\,)}:=
\max_{j=0,1}\{\,\left\|\, T\,\right\|_{\,\bo(\,\calH_j\,;\calK_j\,)}\,\}.
\end{equation}
If $\left\|\,T\,\right\|_{\,\bo(\,\overline{\calH}\,;\,\overline{\calK}\,)}\le 1$ we speak
of a \textit{contraction} from $\overline{\calH}$ to $\overline{\calK}$.

A Banach space $X$ such that $\Delta\subset X\subset \Sigma$ (continuous inclusions) is called
\textit{intermediate} with respect to the pair $\overline{\calH}$.

Let $X$, $Y$ be intermediate spaces with respect to couples $\overline{\calH}$, $\overline{\calK}$, respectively. We say that
$X$, $Y$ are (relative) \textit{interpolation spaces} if there
is a constant $C$ such that $T:\overline{\calH}\to\overline{\calK}$ implies that $T:X\to Y$ and
\begin{equation}\label{C-int}\left\|\,T\,\right\|_{\,\bo(X;Y)}\le C\left\|\,T\,\right\|_{\,\bo(\overline{\calH};\overline{\calK})}.\end{equation}
In the case when $C=1$ we speak about \textit{exact interpolation}.
When $\overline{\calH}=\overline{\calK}$ and $X=Y$ we simply say that $X$ is an (exact) interpolation
space with respect to $\overline{\calH}$.

\smallskip

Let
$H$ be a suitable function of two positive variables and $X$, $Y$ spaces intermediate to the couples $\overline{\calH}$, $\overline{\calK}$, respectively.
We say that the spaces $X$, $Y$ are of \textit{type $H$} (relative to
$\overline{\calH}$, $\overline{\calK}$)
if for any
positive numbers $M_0$, $M_1$ we have
\begin{equation}\label{Htype}\left\|\, T\,\right\|_{\,\bo\left(\calH_i;\calK_i\right)}\le M_i,\quad i=0,1
\qquad \text{implies}\qquad \left\|\,T\,\right\|_{\,\bo(X;Y)}\le H(M_0,M_1).
\end{equation}
The case $H(x,y)=\max\{x,y\}$ corresponds to exact interpolation, while
$H(x,y)=x^{\,1-\theta}y^{\,\theta}$ corresponds to the convexity estimate
\begin{equation}\label{expo}\left\|\,T\,\right\|_{\,\bo(X;Y)}\le
\left\|\,T\,\right\|_{\,\bo(\calH_0;\calK_0)}^{\,\,1-\theta}\,\left\|\,T\,\right\|_{\,\bo(\calH_1;\calK_1)}^{\,\,\theta}.
\end{equation}
In the situation of \eqref{expo}, one says that the interpolation spaces $X$, $Y$ are of \textit{exponent $\theta$} with respect to the
pairs $\overline{\calH}$, $\overline{\calK}$.

\subsection{$K$-spaces} \label{mex} Given a regular Hilbert couple $\overline{\calH}$ and a positive
Radon measure $\vro$ on the compactified half-line $[0,\infty]$ we define an intermediate
quadratic norm by
\begin{equation}\label{*def}\left\|\,x\,\right\|_*^{\, 2}=\left\|\,x\,\right\|_\vro^{\, 2}=\int_{[0,\infty]}\left(1+t^{-1}\right)
K\left(t,x;\,\overline{\calH}\,\right)\, d\vro(t).\end{equation}
Here the integrand $k(t)=\left(1+t^{-1}\right)K(t,x)$ is defined at the points $0$ and $\infty$ by
$k(0)=\left\|\,x\,\right\|_1^{\,2}$ and $k(\infty)=\left\|\,x\,\right\|_0^{\,2}$; we shall write $\calH_*$ or $\calH_\vro$ for the Hilbert space defined by the norm \eqref{*def}.

Let $T\in \bo\left(\,\overline{\calH};\overline{\calK}\,\right)$ and suppose that
$\left\|\,T\,\right\|_{\,\bo\left(\calH_i;\calK_i\right)}\le M_i$; then
\begin{equation}\label{contr}K\left( t,Tx;\,\overline{\calK}\,\right)\le M_0^{\,2}\,
K\left( M_1^{\,2}t/M_0^{\,2},\, x;\,\overline{\calH}\,\right),\quad x\in\Sigma.\end{equation}
In particular, $M_i\le 1$ for $i=0,1$ implies
$\left\|\, Tx\,\right\|_{\,\calK_\vro}\le \left\|\,x\,\right\|_{\,\calH_\vro}$ for all $x\in\calH_\vro$.
It follows that the spaces $\calH_\vro$, $\calK_\vro$ are exact interpolation spaces
with respect to $\overline{\calH}$, $\overline{\calK}$.

\subsubsection*{Geometric interpolation}
When the measure $\vro$ is given by
$$d\vro(t)=c_\theta \frac {t^{-\theta}}{1+t}\, dt,\qquad c_\theta=\frac \pi {\sin\theta\pi},\quad
0<\theta<1,$$
we denote the norm \eqref{*def} by
\begin{equation}\label{mc}\left\|\,x\,\right\|_\theta^{\,2}:=c_\theta
\int_0^\infty t^{-\theta}K\left( t,\, x\right)\, \frac {dt} t.\end{equation}
The corresponding space $\calH_\theta$ is
easily seen to be of exponent $\theta$ with respect to $\overline{\calH}$.
In \textsection \ref{quadro}, we will recognize $\calH_\theta$ as the geometric interpolation space
 which has been studied independently by several authors, see \cite{Mc,U,LM}.

\subsection{Pick functions} \label{opint} Let $\overline{\calH}$ be a regular Hilbert couple.
The squared norm $\left\|\,x\,\right\|_1^{\,2}$ is a densely defined quadratic form
in $\calH_0$, which we represent as
$$\left\|\,x\,\right\|_1^{\,2}=\left\langle\,Ax\,,\,x\,\right\rangle_0=\|\, A^{\,1/2}x\,\|_0^{\, 2}$$
where $A$ is a densely defined, positive, injective (perhaps unbounded) operator in $\calH_0$. The domain
of the positive square-root $A^{1/2}$ is $\Delta$.

\begin{lem} \label{kcalc} We have
in terms of the functional calculus in
$\calH_0$
\begin{equation}\label{int}K\left(t,x\,\right)=\left\langle\,\frac {tA} {1+tA}\, x\, ,\, x\,\right\rangle_0,\quad t>0.\end{equation}
\end{lem}

In the formula \eqref{int}, we have identified the bounded operator $\frac {tA}{1+tA}$ with its extension to $\calH_0$.

\begin{proof} Fix $x\in \Delta$. By a straightforward convexity argument, there is a unique
decomposition $x=x_{0,t}+x_{1,t}$ which is \textit{optimal} in the sense that
\begin{equation}\label{opt}K(t,x)=\left\|\,x_{0,t}\,\right\|_0^{\,2}+t\left\|\,x_{1,t}\,\right\|_1^{\,2}.
\end{equation}
It follows that
$x_{i,t}\in\Delta$ for $i=0,1$. Moreover, for all $y\in\Delta$ we have
$$\frac d {d\epsilon}\{\,\left\|\,x_{0,t}+\epsilon y\,\right\|_0^{\,2}+t\left\|\, x_{1,t}-\epsilon y\,\right\|_1^{\,2}\,\}|_{\epsilon=0}=0,$$
i.e.,
$$\langle\,A^{-1/2}x_{0,t}-tA^{1/2}x_{1,t}\, ,\, A^{1/2}y\,\rangle_0=0,\qquad y\in\Delta.$$
By regularity, we conclude that $A^{-1/2}x_{0,t}=tA^{1/2}x_{1,t}$, whence
\begin{equation}\label{theserel}x_{0,t}=\frac {tA}{1+tA}\, x\qquad \text{and}\qquad x_{1,t}=\frac 1 {1+tA}\, x.
\end{equation}
(Note that the operators in \eqref{theserel} extend to bounded operators on $\calH_0$.)
Inserting the relations \eqref{theserel} into \eqref{opt}, one finishes the proof of the lemma.
\end{proof}

Now fix a positive Radon measure $\vro$ on $[0,\infty]$. The norm in
the space $\calH_\vro$ (see \eqref{*def}) can be written
\begin{equation}\label{fvro} \left\|\,x\,\right\|_\vro^{\,2}=\left\langle\, h(A)x\, ,\, x\,\right\rangle_0,\end{equation}
where
\begin{equation}\label{Pick}h(\lambda)=\int_{[0,\infty]}\frac {(1+t)\lambda} {1+t\lambda}\, d\vro(t).\end{equation}
The class of functions representable in this form for some positive Radon measure $\vro$ is the
class $P'$ of \textit{Pick functions, positive and regular on $\R_+$}.

Notice that
for the definition \eqref{fvro} to make sense, we just need $h$ to be defined on
$\sigma(A)\setminus\{0\}$, where $\sigma(A)$ is the spectrum of $A$.
(The value $h(0)$ is irrelevant since
$A$ is injective).

A calculus exercise shows that for the space $\calH_\theta$ (see \eqref{mc}) we have
\begin{equation}\left\|\, x\,\right\|_\theta^{\, 2}=\langle\, A^\theta x\, ,\, x\,\rangle_0.\end{equation}

\subsection{Quadratic interpolation norms} \label{donolemm} Let $\calH_*$ be any \textit{quadratic} intermediate space relative to
$\overline{\calH}$. We write
$$\left\|\,x\,\right\|_*^{\,2}=\left\langle\, Bx\, ,\, x\,\right\rangle_0$$ where
$B$ is a positive injective operator in $\calH_0$ (the domain of $B^{1/2}$ is $\Delta$).

For a map $T\in \bo(\overline{\calH})$ we shall often use the simplified notations
$$\|\,T\,\|=\|\,T\,\|_{\,\bo\left(\calH_0\right)}\quad ,\quad
\|\,T\,\|_A=\|\, T\,\|_{\,\bo\left(\calH_1\right)}\quad ,\quad \|\,T\,\|_B=\|\,T\,\|_{\,\bo\left(\calH_*\right)}.$$
The reader can check the identities
$$\left\|\,T\,\right\|_A=\|\,A^{1/2}TA^{-1/2}\,\|\quad \text{and}\quad
\left\|\,T\,\right\|_B=\|\,B^{1/2}TB^{-1/2}\,\|.$$
We shall refer to the following lemma as \textit{Donoghue's lemma}, cf. \cite[Lemma 1]{D2}.

\begin{lem} \label{donogh} If $\calH_*$ is exact interpolation with respect to $\overline{\calH}$, then
$B$ commutes with every projection which commutes with $A$ and
$B=h(A)$ where $h$ is some positive Borel function on $\sigma(A)$.
\end{lem}

\begin{proof} For an orthogonal projection $E$ on $\calH_0$, the condition
$\left\|\,E\,\right\|_A\le 1$ is equivalent to that $EAE\le A$, i.e., that
$E$ commutes with $A$. The hypothesis that $\calH_*$ be exact interpolation
thus implies that every spectral projection of $A$ commutes with $B$.
It now follows from von Neumann's bicommutator theorem that $B=h(A)$ for some
positive Borel function $h$ on $\sigma(A)$.
\end{proof}

In view of the lemma, the characterization of the exact quadratic interpolation norms
of a given type $H$ reduces to the characterization of functions
$h:\sigma(A)\to\mathbf{R}_+$ such that for all $T\in \bo\left(\,\overline{\calH}\,\right)$
\begin{equation}\left\|\,T\,\right\|\le M_0\quad \text{and}\quad \left\|\,T\,\right\|_A\le M_1
\quad \Rightarrow\quad \left\|\, T\,\right\|_{h(A)}\le H(M_0,M_1),\end{equation}
or alternatively,
\begin{equation}{T}^*T\le M_0^{\,2}\quad \text{and}\quad {T}^*AT\le M_1^{\,2}\,A\quad \Rightarrow\quad
{T}^*h(A)T\le H(M_0,M_1)^{\,2}\, h(A).\end{equation}
The set of functions $h:\sigma(A)\to\mathbf{R}_+$ satisfying these equivalent conditions forms a convex cone $C_{H,A}$; its elements are called \textit{interpolation functions of type $H$ relative to $A$}.
In the case when $H(x,y)=\max\{x,y\}$ we simply write $C_A$ for $C_{H,A}$ and speak of \textit{exact interpolation
functions relative to $A$}.

\subsection{Exact Calderón pairs and the $K$-property} \label{capa} Given two intermediate normed spaces $Y$, $X$ relative to
$\overline{\calH}$, $\overline{\calK}$, we say that they are (relatively) \textit{exact
$K$-monotonic} if the conditions
$$x^0\in X\quad \text{and}\quad K\left( t,y^0;\,\overline{\calH}\,\right)\le K\left(t,x^0;\,\overline{\calK}\,\right),\qquad t>0$$
imply that
$$y^0\in Y\quad\text{and}\quad \|\,y^0\,\|_Y\le\|\,x^0\,\|_X.$$

It is easy to see that \textit{exact $K$-monotonicity implies exact interpolation}.

\begin{proof}[Proof of this.] If $\left\|\,T\,\right\|_{\,\bo(\overline{\calK};\overline{\calH})}\le 1$
then $\forall x,t$: $K\left(t,Tx;\,\overline{\calH}\,\right)\le K\left(t,x;\,\overline{\calK}\,\right)$ whence
$\left\|\,Tx\,\right\|_Y\le \left\|\,x\,\right\|_X$, by exact $K$-monotonicity. Hence  $\left\|\,T\,\right\|_{\,\bo(X;Y)}\le 1$.
\end{proof}

Two pairs $\overline{\calH}$, $\overline{\calK}$ are called
\textit{exact relative Calderón pairs}
if any two exact interpolation (Banach-) spaces $Y$, $X$ are exact $K$-monotonic.
Thus, with respect to to exact Calderón pairs, exact interpolation is equivalent to exact $K$-monotonicity.
The term "Calderón pair" was coined after thorough investigation of A. P. Calderón's study of the pair $\left(L_1,L_\infty\right)$, see \cite{C1} and \cite{Cw}.

In our present discussion, it is not convenient to work directly with the definition of exact Calderón pairs. Instead, we
shall use the following, closely related notion.

We say that a pair of couples $\overline{\calH}$, $\overline{\calK}$ has the \textit{relative (exact) \Kpr}
if for all $x^0\in\Sigma(\,\overline{\calK}\,)$ and $y^0\in\Sigma(\,\overline{\calH}\,)$ such that
\begin{equation}\label{below}K\left(t,y^0;\,\overline{\calH}\,\right)\le K\left(t,x^0;\,\overline{\calK}\,\right),\quad t>0,\end{equation} there exists a map $T\in \bo(\,\overline{\calK};\overline{\calH}\,)$
such that $Tx^0=y^0$ and $\left\|\,T\,\right\|_{\,\bo(\overline{\calK};\overline{\calH})}\le 1$.

\begin{lem} \label{caldp} If $\overline{\calH}$, $\overline{\calK}$ have the relative \Kpr, then they are
exact relative Calderón pairs.
\end{lem}

\begin{proof} Let $Y$, $X$ be exact interpolation spaces relative to $\overline{\calH}$, $\overline{\calK}$
and take $x^0\in X$ and $y^0\in\Sigma(\,\overline{\calH}\,)$ such that \eqref{below} holds. By the \Kpr there is $T:\overline{\calK}\to\overline{\calH}$ such that
$Tx^0=y^0$ and $\left\|\,T\,\right\|\le 1$. Then $\left\|\,T\,\right\|_{\,\bo(X;Y)}\le 1$, and so
$\|\,y^0\,\|_Y =\|\,Tx^0\,\|_Y\le \|\,x^0\,\|_X$. We have shown that $Y$, $X$ are exact $K$-monotonic.
\end{proof}

In the diagonal case $\overline{\calH}=\overline{\calK}$, we simply say that $\overline{\calH}$ is
an \textit{exact Calderón couple} if for intermediate spaces $Y,X$, the
property of being exact interpolation is equivalent to being exact $K$-monotonic. Likewise, we say that $\overline{\calH}$
has the \textit{\Kpr} if the pair of couples $\overline{\calH}$, $\overline{\calH}$ has that property.

\begin{rem} \label{simrem}
For an operator $T:\overline{\calK}\to\overline{\calH}$ to be a contraction, it is necessary and sufficient that
\begin{equation}\label{indee}K\left(t,Tx;\,\overline{\calH}\,\right)\le K\left(t,x;\,\overline{\calK}\,\right),\qquad
x\in \Sigma(\,\overline{\calK}\,),\, t>0.\end{equation}
Indeed, the necessity is immediate. To prove the sufficiency it suffices to observe that letting $t\to \infty$ in \eqref{indee} gives
$\left\|\, Tx\,\right\|_0\le\left\|\,x\,\right\|_0$, and dividing \eqref{indee} by $t$, and then letting $t\to 0$, gives
that $\left\|\, Tx\,\right\|_1\le\left\|\,x\,\right\|_1$.
\end{rem}

\section{Mapping properties of Hilbert couples} \label{chap2}

\subsection{Main results}
We shall elaborate on the following main result from \cite{A2}.

\begin{mth} \label{mthm} Any pair of regular Hilbert couples $\overline{\calH}$, $\overline{\calK}$ has the
relative \Kpr.
\end{mth}

Before we come to the proof of Theorem \ref{mthm},
we note some consequences of it. We first have the following corollary, which shows that a strong
form of the $K$-property is true.

\begin{cor} \label{mcor} Let $\overline{\calH}$ be a regular Hilbert couple and $x^0,y^0\in\Sigma$ elements
such that
\begin{equation}\label{snorm}K\left(t,y^{0}\right)\le M_0^{\,2}\, K\left(M_1^{\,2}t/M_0^{\,2}\, ,\,x^{0}\right),\qquad t>0.\end{equation}
Then
\begin{enumerate}
\item[(i)] There exists a map $T\in \bo\left(\,\overline{\calH}\,\right)$ such that
$Tx^{0}=y^{0}$ and $\left\|\,T\,\right\|_{\,\bo\left(\calH_i\right)}\le M_i$, $i=0,1$.
\item[(ii)] If $x^{0}\in X$ where $X$ is an interpolation space of type $H$, then
$$\|\,y^{0}\,\|_X\le H\left(M_0,M_1\,\right)\,\|\,x^{0}\,\|_X.$$
\end{enumerate}
\end{cor}

\begin{proof} (i) Introduce a new couple $\overline{\calK}$ by letting $\left\|\,x\,\right\|_{\calK_i}=M_i\|\,x\,\|_{\calH_i}$.
The relation \eqref{snorm} then says that
$$K\left(t,y^{0};\overline{\calH}\right)\le K\left(t,x^{0};\overline{\calK}\right),\qquad t>0.$$
By Theorem \ref{mthm} there is a contraction $T:\overline{\calK}\to\overline{\calH}$ such that
$Tx^{0}=y^{0}$.
It now suffices to note that
$\left\|\,T\,\right\|_{\,\bo\left(\calH_i\right)}=M_i\left\|\,T\,\right\|_{\,\bo\left(\calK_i;\calH_i\right)}$; (ii) then follows from Lemma \ref{caldp}.
\end{proof}

We next mention some equivalent versions of Theorem \ref{mthm}, which uses the families of functionals
$K_p$ and $E_p$ defined (for $p\ge 1$ and $t,s>0$) via
\begin{equation}\label{rem?}\begin{split}
K_p(t)&=K_p(t,x)=K_p\left(t,x;\overline{\calH}\right)=\inf_{x=x_0+x_1}\left\{\,\left\|\,x_0\,\right\|_0^{\,p}+t\left\|\,x_1\,\right\|_1^{\,p}\,\right\}\\
E_p(s)&=E_p(s,x)=E_p\left(s,x;\overline{\calH}\right)=\inf_{\left\|\,x_0\,\right\|_0^{\, p}\le s}\left\{\,\left\|\,x-x_0\,\right\|_1^{\,p}\,\right\}.\\
\end{split}\end{equation}
Note that $K=K_2$ and that $E_p(s)=E_1\left(s^{1/p}\right)^{\,p}$; the $E$-functionals are used in
approximation theory. One has that $E_p$ is decreasing and convex on $\R_+$ and that
$$K_p(t)=\inf_{s>0}\left\{\, s+tE_p(s)\,\right\},$$
which means that $K_p$ is a kind of \textit{Legendre transform} of $E_p$.
The inverse
Legendre transformation
takes the form
$$E_p(s)=\sup_{t>0}\left\{\,\frac {K_p(t)} t-\frac s t\,\right\}.$$
It is now immediate that, for all $x\in\Sigma\left(\,\overline{\calK}\,\right)$ and
$y\in\Sigma\left(\,\overline{\calH}\,\right)$, we have
\begin{equation}\label{stf}K_p(t,y)\le K_p(t,x),\quad t>0\qquad \Leftrightarrow\qquad E_p(s,y)\le E_p(s,x),\quad s>0.\end{equation}
Since moreover $E_p(s)=E_2\left(s^{2/p}\right)^{\,p/2}$, the conditions in \eqref{stf} are equivalent to that
$K(t,y)\le K(t,x)$ for all $t>0$.
We have shown the following result.

\begin{cor} In Theorem \ref{mthm}, one can substitute the $K$-functional for any of the functionals
$K_p$ or $E_p$.
\end{cor}

Define an exact interpolation norm $\|\cdot\|_{\vro,p}$ relative to
$\overline{\calH}$ by
$$\left\|\,x\,\right\|_{\vro,p}^{\,p}=\int_{[0,\infty]}\left(1+t^{-1}\right)K_p(t,x)\, d\vro(t)$$
where $\vro$ is a positive Radon measure on $[0,\infty]$. This norm is non-quadratic when $p\ne 2$, but is of course equivalent to the quadratic norm corresponding to $p=2$.

\subsection{Reduction to the diagonal case} It is not hard to reduce
the discussion of Theorem \ref{mthm} to a diagonal situation.

\begin{lem} \label{diagred} If the \Kpr holds for regular Hilbert couples in the diagonal case $\overline{\calH}=\overline{\calK}$, then it
holds in general.
\end{lem}

\begin{proof} Fix elements $y^{0}\in\Sigma(\,\overline{\calH}\,)$ and
$x^{0}\in\Sigma(\,\overline{\calK}\,)$ such that the inequality \eqref{below} holds.
We must
construct a map $T:\overline{\calK}\to\overline{\calH}$
such that $Tx^{0}=y^{0}$ and $\left\|\,T\,\right\|\le 1$.

To do this, we form the direct sum $\overline{\calS}=\left(\calH_0\oplus\calK_0,\calH_1\oplus\calK_1\right)$.
It is clear that $\calS_0+\calS_1=\left(\calH_0+\calH_1\right)\oplus\left(\calK_0+\calK_1\right)$,
and that
$$K\left(t,x\oplus y;\overline{\calS}\right)=K\left(t,x;\overline{\calH}\right)+K\left(t,y;\overline{\calK}
\right).$$
Then
$$K\left(t,0\oplus y^{0};\overline{\calS}\right)\le K\left(t,x^{0}\oplus 0;\overline{\calS}\right).$$
Hence assuming that the couple $\overline{\calS}$ has the \Kpr, we
can assert the existence of a map $S\in \bo(\,\overline{\calS}\,)$ such that
$S(x^{0}\oplus 0)=0\oplus y^{0}$ and $\left\|\,S\,\right\|\le 1$. Letting
$P:\calS_0+\calS_1\to \calK_0+\calK_1$ be the orthogonal projection, the assignment
$Tx=PS(x\oplus 0)$ now defines a map such that $Tx^{0}=y^{0}$ and $\left\|\,T\,\right\|_{\,\bo(\overline{\calH};\overline{\calK})}\le 1$.
\end{proof}

\subsection{The principal case} \label{tml} The core content of Theorem \ref{mthm} is contained in the following statement.

\begin{thm} \label{mlem} Suppose that a regular Hilbert couple $\overline{\calH}$ is finite dimensional and that all eigenvalues of the corresponding operator $A$ are of unit multiplicity.
Then $\overline{\calH}$ has the \Kpr.
\end{thm}

We shall settle for proving Lemma \ref{mlem} in this section, postponing to Section \ref{appa} the general case of Theorem \ref{mthm}.

To prepare for the proof, we write the eigenvalues $\lambda_i$ of $A$ in increasing order,
$$\sigma(A)=\{\lambda_i\}_1^n\quad\text{where}\quad 0<\lambda_1<\cdots<\lambda_n.$$
Let $e_i$ be corresponding eigenvectors of unit length for the
norm of $\calH_0$. Then for a vector $x=\sum x_ie_i$ we have
$$\|\,x\,\|_0^{\,2}=\sum_1^n|x_i|^{\,2}\qquad ,\qquad \|\,x\,\|_1^{\,2}=\sum_1^n\lambda_i|x_i|^{\,2}.$$
Working in the coordinate system $(e_i)$,
the couple $\overline{\calH}$ becomes identified with the
$n$-dimensional weighted $\ell_2$ couple
$$\overline{\ell_2^n}(\lambda):=\left(\ell_2^n,\ell_2^n(\lambda)\right),$$
where we write $\lambda$ for the sequence $(\lambda_i)_1^n$.

We will henceforth identify a vector
$x=\sum x_ie_i$ with the point $x=(x_i)_1^n$ in $\C^{\,n}$; accordingly, the space
$\bo\left(\ell_2^n\right)$ is identified with the $C^*$-algebra $M_n(\C)$ of complex $n\times n$ matrices.

It will be convenient to reparametrize the $K$-functional for the couple $\overline{\ell_2^n}(\lambda)$ and write
\begin{equation}\label{clam}\kfun_\lambda(t,x):=K\left(1/t,x;\,\overline{\ell_2^n}(\lambda)\,\right).\end{equation}
By Lemma \ref{kcalc} we have
\begin{equation}\label{clam2}\kfun_\lambda(t,x)=\sum_{i=1}^n\frac {\lambda_i}{t+\lambda_i}|x_i|^{\,2},\qquad x\in\C^{\,n}.\end{equation}

\subsection{Basic reductions} To prove that the couple $\overline{\ell_2^n}(\lambda)$ has the
\Kpr,
we introduce an auxiliary parameter $\rho>1$. The exact value of $\rho$ will change meaning during
the course of the argument, the main point being that it can be chosen arbitrarily close to $1$.

Initially, we pick any $\rho>1$ such that $\rho\lambda_i<\lambda_{i+1}$ for all $i$;
we assume also that we are given two elements $x^{0},y^{0}\in\C^{\,n}$ such that
\begin{equation}\label{assume}\kfun_\lambda\left(t,y^{0}\right)<\frac 1 \rho \,\kfun_\lambda\left(t,x^{0}\right),\qquad t\ge 0.\end{equation}
We must
construct a matrix $T\in M_n(\C)$ such that
\begin{equation}\label{conclude}Tx^{0}=y^{0}\quad \text{and}\quad \kfun_\lambda\left(t,Tx\right)\le \kfun_\lambda\left(t,x\right), \quad x\in \C^{\,n},\,t>0.\end{equation}

Define $\tilde{x}^{0}=(|x^{0}_i|)_1^n$ and
$\tilde{y}^{0}=(|y^{0}_i|)_1^n$ and suppose that
$$\kfun_\lambda(t,\tilde{y}^{0})<\frac 1 \rho\, \kfun_\lambda(t,\tilde{x}^{0}),\qquad t\ge 0.$$
Suppose that we can find an operator $T_0\in M_n(\C)$ such that
$T_0\tilde{x}^{0}=\tilde{y}^{0}$ and $\kfun_\lambda\left(t,T_0x\right)<\kfun_\lambda(t,x)$ for all $x\in\C^{\,n}$ and $t>0$. Writing
$x^{0}_k=e^{i\theta_k}\tilde{x}^{0}_k$ and $y^{0}_k=e^{i\fii_k}\tilde{y}^{0}_k$
where $\theta_k,\fii_k\in\mathbf{R}$, we then have $Tx^{0}=y^{0}$ and
$\kfun_\lambda\left(t,Tx\right)< \kfun_\lambda(t,x)$  where
$$T=\diag(e^{i\fii_k})T_0\diag(e^{-i\theta_k}).$$
Replacing $x^0,y^0$ by $\tilde{x}^0$, $\tilde{y}^0$ we
can thus assume that the coordinates $x^{0}_i$ and $y^{0}_i$ are non-negative; replacing them by small
perturbations if necessary, we can assume that they are strictly positive, at the expense of slightly
diminishing the number $\rho$.

\smallskip

Now put $\beta_i=\lambda_i$ and $\alpha_i=\rho\lambda_i$. Our assumption on $\rho$ means that
$$0<\beta_1<\alpha_1<\cdots<\beta_n<\alpha_n.$$
Using the explicit expression for the $K$-functional, it is plain to check that
$$\kfun_\beta(t,x)\le \kfun_\alpha(t,x)\le\rho \kfun_\beta(t,x),\qquad x\in\C^n,\, t\ge 0.$$
Our assumption \eqref{assume} therefore implies that
\begin{equation}\label{assume2}\kfun_\alpha(t,y^{0})<\kfun_\beta(t,x^{0}),\quad t\ge 0.\end{equation}
We shall verify the existence of a matrix $T=T_\rho=T_{\rho,x^{0},y^{0}}$ such that
\begin{equation}\label{conclude2}Tx^{0}=y^{0}\quad\text{and}\quad \kfun_\alpha\left(t,Tx\right)\le \kfun_\beta\left(t,x\right),\qquad x\in\C^n,\,t>0.\end{equation}
It is clear by compactness that, as $\rho\downarrow 1$, the corresponding matrices $T_\rho$ will cluster at
some point $T$ satisfying $Tx^{0}=y^{0}$ and $\left\|\,T\,\right\|_{\,\bo(\,\overline{\calH}\,)}\le 1$.
(See Remark \ref{simrem}.)

\smallskip

In conclusion, the proof of Theorem \ref{mlem} will be complete when we can construct
a matrix $T$ satisfying \eqref{conclude2} with $\rho$ arbitrarily close to $1$.

\subsection{Construction of $T$} Let $\calP_k$ denote the linear space of complex polynomials of degree at most $k$.
We shall use the polynomials
$$L_\alpha(t)=\prod_1^n \left(t+\alpha_i\right)\quad ,\quad L_\beta(t)=\prod_1^n\left(t+\beta_i\right),$$
and the product $L=L_\alpha L_\beta$. Notice that
\begin{equation}\label{line}L'(-\alpha_i)<0\quad,\quad L'(-\beta_i)>0.\end{equation}
Recalling the formula \eqref{clam2}, it is clear that we can
define a real polynomial $P\in\calP_{2n-1}$ by
\begin{equation}\label{Pdef}\frac {P(t)}{L(t)}=\kfun_\beta\left(t,x^{0}\right)-\kfun_\alpha\left(t,y^{0}\right).\end{equation}
Clearly $P(t)> 0$ when $t\ge 0$. Moreover, a consideration of the residues at the poles of the right-hand member shows that $P$ is uniquely defined by the values
\begin{equation}\label{tuck}P\left(-\beta_i\right)=(x_i^{0})^{\,2}\beta_iL'\left(-\beta_i\right)\quad,\quad P\left(-\alpha_i\right)=-(y_i^{0})^{\,2}\alpha_iL'\left(-\alpha_i\right).\end{equation}
Combining with \eqref{line}, we conclude that
\begin{equation}\label{pine}P\left(-\alpha_i\right)>0\quad \text{and}\quad P\left(-\beta_i\right)>0.\end{equation}
Perturbing the problem slightly, it is clear that we can assume that $P$ has exact degree $2n-1$,
and that all zeros of $P$ have multiplicity $1$. (We here diminish the value of $\rho>1$ somewhat, if necessary.)

Now, $P$ has $2n-1$ simple zeros, which we split according to
$$P^{-1}\left(\left\{0\right\}\right)=\left\{-r_i\right\}_{i=1}^{2m-1}\cup\left\{-c_i,-\bar{c}_i\right\}_{i=1}^{n-m},$$
where the $r_i$ are positive and the $c_i$ are non-real, and chosen to have
positive imaginary parts.
The following is the key observation.

\begin{lem} We have that
\begin{equation}\label{one}L'\left(-\beta_i\right)P\left(-\beta_i\right)>0\quad,\quad L'\left(-\alpha_i\right)P\left(-\alpha_i\right)<0\end{equation}
and there is a splitting $\left\{r_i\right\}_{i=1}^{2m-1}=\left\{\delta_i\right\}_{i=1}^m\cup
\left\{\gamma_i\right\}_{i=1}^{m-1}$ such that
\begin{equation}\label{two}L\left(-\delta_j\right)P'\left(-\delta_j\right)>0\quad ,\quad L\left(-\gamma_k\right)P'\left(-\gamma_k\right)<0.\end{equation}
\end{lem}

\begin{proof} The inequalities \eqref{one} follow immediately from
\eqref{pine} and \eqref{line}. It remains to prove \eqref{two}.

Let $-h$ denote the leftmost real zero of the polynomial $LP$ (of degree $4n-1$).
We claim
that $P(-h)=0$. If this were not the case, we would have $h=\alpha_n$. Since the degree of $P$
is odd, $P(-t)$ is negative for large values of $t$, and so $P(-\alpha_n)<0$ contradicting \eqref{pine}.
We have shown that $P(-h)=0$. Since all zeros of $LP$ have multiplicity $1$, we have
$(LP)'(-h)\ne 0$, whence
$$L(-h)P'(-h)=(LP)'(-h)>0.$$
We write $\delta_m=h$ and put $P_*(t)=P(t)/(t+\delta_m)$. Since $t+\delta_m>0$
for
$t\in\{-\alpha_i,-\beta_i\}_1^n$, we have by \eqref{pine} that for all $i$
$$P_*(-\alpha_i)>0\quad \text{and}\quad P_*(-\beta_i)>0.$$
Denote by
$\{-{r_j}^*\}_{j=1}^{2m-2}$ the real zeros of $P_*$. Since the degree of $LP_*$ is even and the polynomial
$(LP_*)'$ has alternating signs in the set
$\{-\alpha_i,-\beta_i\}_{i=1}^n\cup\{-{r_i}^*\}_{i=1}^{2m-2}$, we can split the zeros of $P_*$ as
$\{-\delta_i,-\gamma_i\}_{i=1}^{m-1}$, where
\begin{equation}\label{pschine}L(-\delta_i)P_*'(-\delta_i)>0\quad ,\quad L(-\gamma_i)P_*'(-\gamma_i)<0.\end{equation}
Since $P'(-{r_j}^*)=(\delta_m-{r_j}^*)P_*'(-{r_j}^*)$ and $\delta_m>{r_j}^*$, the signs of $P'(-{r_j}^*)$ and
$P_*'(-{r_j}^*)$ are equal, proving \eqref{two}.
\end{proof}

Recall that $\{-c_i\}_1^{n-m}$ denote the zeros of $P$ such that $\im c_i>0$. We put
(with the convention that an empty product equals $1$)
$$L_\delta(t)=\prod_{i=1}^m(t+\delta_i)\quad ,\quad L_\gamma(t)=\prod_{i=1}^{m-1}(t+\gamma_i)
\quad ,\quad L_c(t)=\prod_{i=1}^{n-m}(t+c_i).$$

We define a linear map $F:\C^{n+m}\to\C^{n+m-1}$ in the following way.
First define a subspace $U\subset \calP_{2n-1}$ by
$$U=\left\{\,L_cq\,;\, q\in \calP_{n+m-1}\,\right\}.$$

Notice that $U$ has dimension $n+m-1$ and that
$P\in U$; in fact $P=aL_cL_c^{\,*}L_\delta L_\gamma$ where $a$ is the leading coefficient
and the $*$-operation is defined by $L^{\,*}(z)=\overline{L(\bar{z})}$.

For a polynomial $Q\in U$ we have
\begin{equation}\label{sasplit}\begin{split}
\frac {\left|\,Q(t)\,\right|^{\,2}}{L(t)P(t)}&=
\sum_{i=1}^n|x_i|^{\,2}\frac {\beta_i} {t+\beta_i}+\sum_{i=1}^n|x_i'|^{\,2}\frac {\delta_i}{t+\delta_i}\\
&-\sum_{i=1}^n|y_i|^{\,2}\frac {\alpha_i} {t+\alpha_i}-\sum_{i=1}^{m-1}|y_i'|^{\,2}\frac {\gamma_i}{t+\gamma_i},\\
\end{split}
\end{equation}
where, for definiteness,
\begin{eqnarray}
\label{zwolf}x_i&=\dfrac {Q(-\beta_i)} {\sqrt{\,\beta_iL'(-\beta_i)P(-\beta_i)}}\quad\,\, ; \quad\,\,
x_j'&=\frac {Q(-\delta_j)} {\sqrt{\,\delta_jL'(-\delta_j)P(-\delta_j)}}\\
\label{dreiz}y_i&=\dfrac {Q(-\alpha_i)} {\sqrt{\,-\alpha_iL'(-\alpha_i)P(-\alpha_i)}}\quad ; \quad
y_j'&=\frac {Q(-\gamma_j)} {\sqrt{\,-\gamma_jL'(-\gamma_j)P(-\gamma_j)}}.
\end{eqnarray}

The identities in \eqref{zwolf} give rise to a linear map
\begin{equation}\label{L1}M:\C^{\,n}\oplus\C^{\,m}\to U\quad ;\quad \left[x;x'\right]\mapsto Q.\end{equation}
We can similarly regard \eqref{dreiz} as a linear map
\begin{equation}\label{L2}N:U\to\C^{\,n}\oplus\C^{\,m-1}\quad ;\quad Q\mapsto \left[y;y'\right].\end{equation}
Our desired map $F$ is defined as the composite
$$F=N M:\C^{\,n}\oplus\C^{\,m}\to\C^{\,n}\oplus\C^{\,m-1}\quad ;\quad [x;x']\mapsto [y;y'].$$ Notice that if
$Q=M\left[x;x'\right]$ and $\left[y;y'\right]=F\left[x;x'\right]$ then \eqref{sasplit} means that
$$\kfun_{\beta\oplus\delta}\left(t,\left[x; x'\right]\right)-\kfun_{\alpha\oplus\gamma}\left(t,F\left[x; x'\right]\right)=\frac
{\left|\,Q(t)\,\right|^{\,2}}{L(t)P(t)}\ge 0,\qquad t\ge 0.$$
This implies that $F$ is a contraction from $\overline{\ell_2^{n+m}}(\beta\oplus\delta)$ to
$\overline{\ell_2^{n+m-1}}(\alpha\oplus\gamma)$.

We now define $T$ as a "compression" of $F$. Namely, let
$E:\C^{\,n}\oplus\C^{\,m-1}\to\C^{\,n}$ be the projection onto the first $n$ coordinates, and define
an operator $T$ on $\C^n$ by
$$Tx=EF\left[x;0\right],\qquad x\in\C^{\,n}.$$
Taking $Q=P$ in \eqref{sasplit} we see that $Tx^{0}=y^{0}$. Moreover,
\begin{align*}\kfun_\beta\left(t,x\right)-\kfun_\alpha\left(t,Tx\right)&=\sum_{i=1}^n|x_i|^{\,2}\frac {\beta_i}{t+\beta_i}-
\sum_{i=1}^n|y_i|^{\,2}\frac {\alpha_i}{t+\alpha_i}\\
&\ge \sum_{i=1}^n|x_i|^{\,2}\frac {\beta_i}{t+\beta_i}-\sum_{i=1}^n|y_i|^{\,2}\frac {\alpha_i}{t+\alpha_i}
-\sum_{j=1}^{m-1}|y_i'|^{\,2}\frac {\gamma_i}{t+\gamma_i}\\
&=\kfun_{\beta\oplus\delta}\left(t,\left[x; 0\right]\right)-\kfun_{\alpha\oplus\gamma}\left(t,F\left[x; 0\right]\right)=\frac{|\,Q(t)\,|^{\,2}}
{L(t)P(t)}.\end{align*}
Since the right-hand side is non-negative, we have shown that
$$\kfun_\alpha\left(t,Tx\right)\le \kfun_\beta(t,x),\quad t>0,\, x\in\C^n,$$
as desired.
The proof of Theorem \ref{mlem} is finished. q.e.d.

\subsection{Real scalars} \label{rescrem} Theorem \ref{mlem} holds also in the case of Euclidean
spaces over the real scalar field.
To see this, assume without loss of generality that the vectors
$x^{0},y^{0}\in\C^n$
have \textit{real entries} (still satisfying $\kfun_\lambda\left(t,y^{0}\right)\le \kfun_\lambda\left(t,x^{0}\right)$ for all $t>0$).

By Theorem \ref{mlem} we can find a (complex) contraction $T$ of $\overline{\ell_2^n}(\lambda)$
such that $Tx^{0}=y^{0}$. It is clear that the operator $T^{\,*}$ defined by
$T^{\,*}x=\overline{T\left(\bar{x}\right)}$ satisfies those same conditions. Replacing
$T$ by $\frac 1 2\left(T+T^{\,*}\right)$ we obtain a real matrix $T\in M_n(\R)$, which is a contraction
of $\overline{\ell_2^n}(\lambda)$
and maps $x^{0}$ to $y^{0}$. $\qed$

\subsection{Explicit representations} We here
deduce an explicit representation for the operator $T$ constructed above.

Let $x^{0}$ and $y^{0}$ be two non-negative vectors such that
$$\kfun_\lambda\left(t,y^{0}\right)\le \kfun_\lambda\left(t,x^{0}\right),\qquad t>0.$$
For small $\rho>0$ we perturb $x^{0}$, $y^{0}$ slightly to vectors $\tilde{x}^{\, 0}$, $\tilde{y}^{0}$ which satisfy
the conditions imposed the previous subsections. We can then construct a matrix $T=T_\rho$
such that
\begin{equation}\label{app}T \tilde{x}^{\, 0}=\tilde{y}^{\, 0}\quad \text{and}\quad \kfun_\alpha\left(t,Tx\right)\le \kfun_\beta\left(t,x\right),\qquad t>0,\, x\in\C^{\,n},\end{equation}
where $\beta=\lambda$ and $\alpha=\rho\lambda$. As $\rho$, $\tilde{x}^{0}$, $\tilde{y}^{\, 0}$ approaches
$1$, $x^{0}$, resp. $y^{0}$, it is clear that any cluster point $T$ of the set of contractions $T_\rho$ will satisfy
$$Tx^{0}=y^{0}\quad \text{and}\quad \kfun_\lambda\left(t,Tx\right)\le \kfun_\lambda\left(t,x\right),\quad t>0,\, x\in\C^{\,n}.$$

\begin{thm} \label{comp} The matrix $T=T_\vro=\left(\tau_{ik}\right)_{i,k=1}^{n}$ where
\begin{equation}\label{calc}
\tau_{ik}=\re\left[\frac 1 {\alpha_i-\beta_k}\frac {\tilde{x}_k^{0}}{\tilde{y}_i^{0}}
\frac {\beta_k L_\delta(-\alpha_i)L_c(-\alpha_i)L_\alpha(-\beta_k)}
{\alpha_iL_\delta(-\beta_k)L_c(-\beta_k)L_\alpha'(-\alpha_i)}\right]
\end{equation}
satisfies \eqref{app}.
\end{thm}

\begin{proof} The range of the map $\C^{\,n}\to U$, $x\mapsto M\left[x;0\right]$ (see \ref{L1}) is precisely the $n$-dimensional subspace
\begin{equation}\label{tsb}V:=L_\delta L_c\cdot\calP_{n-1}=\{L_\delta L_cR;\, R\in \calP_{n-1}\}\subset U.\end{equation}
We introduce a basis $\left(Q_k\right)_{k=1}^n$ for $V$ by
$$Q_k(t)=\frac {L_\delta(t)L_c(t)L_\beta(t)}{t+\beta_k}
\frac {\sqrt{\beta_kL'(-\beta_k)P(-\beta_k)}}{L_\delta(-\beta_k)L_c(-\beta_k)L_\beta'(-\beta_k)}.$$
Then
$$\frac {Q_k(-\beta_i)} {\sqrt{\beta_iL'(-\beta_i)P(-\beta_i)}}=\begin{cases} 1 & i=k,\cr
0 & i\ne k.\cr
\end{cases}$$
Denoting by $(e_i)$ the canonical basis in $\C^{\,n}$ and using \eqref{zwolf}, \eqref{dreiz} we get
\begin{align*}
\tau_{ik}&=(Te_k)_i=\frac {Q_k(-\alpha_i)} {\sqrt{\alpha_iL'(-\alpha_i)P(-\alpha_i)}}\\
&=\frac 1 {\beta_k-\alpha_i}\frac  {L_\delta(-\alpha_i)L_c(-\alpha_i)L_\beta(-\alpha_i)}
 {L_\delta(-\beta_k)L_c(-\beta_k)L_\beta'(-\beta_k)}\left(
 \frac {\beta_kL'(-\beta_k)P(-\beta_k)} {-\alpha_iL'(-\alpha_i)P(-\alpha_i)}\right)^{1/2}.
 \end{align*}
 Inserting the expressions \eqref{tuck} for $P(-\alpha_i)$ and $P(-\beta_k)$ and taking real
 parts (see the remarks in \textsection \ref{rescrem}), we obtain the formula \eqref{calc}.
 \end{proof}

\begin{rem} It is easy to see that, if we pick all matrix-elements real, some elements $\tau_{ik}$ of the matrix $T$ in \eqref{calc}
will be negative, even while the numbers $x^0_i$ and $y^0_k$ are positive. It was proved in \cite{A2},
Theorem 2.3, that this is necessarily so. Indeed, one there constructs an example of a five-dimensional
couple $\overline{\ell_2^{\,5}}(\lambda)$ and two vectors $x^0,y^0\in\R^5$ having non-negative entries
such that \textit{no} contraction $T=\left(\tau_{ik}\right)_{i,k=1}^5$ on $\overline{\ell_2^{\,5}}(\lambda)$
having all matrix entries $\tau_{ik}\ge 0$ can satisfy $Tx^0=y^0$. On the other hand, if one settles for
using a matrix with $\left\|\, T\,\right\|\le\sqrt{2}$, then it is possible to find one with
only non-negative matrix entries. Indeed, such a matrix was used by Sedaev \cite{Sed}, see also \cite{Sp2}.
\end{rem}


\subsection{On sharpness of the norm-bounds}We shall show that if $m<n$ (i.e. if the polynomial $P$ has at least one non-real zero),
then the norm $\left\|\,T\,\right\|_{\,\bo\left(\calH_i\right)}$ of the contraction $T$ constructed above is very close to $1$ for $i=0,1$.

 We first claim that $\left\|\,T\,\right\|_{\,\bo\left(\calH_0\right)}= 1$. To see this, we notice that if $m<n$, then there is a non-trivial polynomial $Q^{(1)}$ in the space $V$ (see \eqref{tsb}) which vanishes
at the points $0,\gamma_1,\ldots,\gamma_{m-1}$. If $x_i^{(1)}$ and $y_i^{(1)}$ are defined by the formulas \eqref{zwolf} and \eqref{dreiz}
(while $(x_j^{(1)})'=(y_k^{(1)})'=0$), we then have $Tx^{(1)}=y^{(1)}$ and
$$\kfun_\beta(t,x^{(1)})-\kfun_\alpha(t,y^{(1)})=\frac{|\,Q^{(1)}(t)\,|^{\,2}}{L(t)P(t)},\qquad t>0.$$
Choosing $t=0$ we conclude that $\|\,x^{(1)}\,\|_{\ell_2^n}^{\,2}-\|\,Tx^{(1)}\,\|_{\ell_2^n}^{\,2}= 0$, whence $\left\|\,T\,\right\|_{\,\bo\left(\calH_0\right)}\ge 1$, proving our claim.

Similarly, the condition $m<n$ implies the existence of a polynomial $Q^{(2)}\in V$ of degree at most $n+m-2$ vanishing at the points $\gamma_1,\ldots,\gamma_{m-1}$.
Constructing vectors $x^{(2)}$, $y^{(2)}$ via \eqref{zwolf} and \eqref{dreiz} we will have $Tx^{(2)}=y^{(2)}$ and
$$\kfun_\beta(t,x^{(2)})-\kfun_\alpha(t,y^{(2)})=\frac{|\,Q^{(2)}(t)\,|^{\,2}}{L(t)P(t)},\qquad t>0.$$
Multiplying this relation by $t$ and then sending $t\to\infty$, we find that $\|\,x^{(2)}\,\|_{\ell_2^n(\beta)}^{\,2}-\|\,Tx^{(2)}\,\|_{\ell_2^n(\alpha)}^{\,2}=0$,
which implies $\left\|\,T\,\right\|_{\,\bo(\calH_1)}\ge \rho^{-1/2}$.

\subsection{A remark on weighted $\ell_p$-couples} \label{remu} As far as we are aware, if $1<p<\infty$ and $p\ne 2$, it is still an open question whether the couple $\overline{\ell_p^n}(\lambda)
 =\left(\ell_p^n,\ell_p^n(\lambda)\right)$
 is an exact Calderón couple or not. (When $p=1$ or $p=\infty$ it is
exact Calderón; see \cite{SS} for the case $p=1$; the case $p=\infty$ is essentially just the
Hahn-Banach theorem.)

It is well known, and easy to prove, that the $K_p$-functional (see \eqref{rem?})
corresponding to the couple $\overline{\ell_p^n}(\lambda)$ is given by the explicit formula
$$K_p\left(t,x;\overline{\ell_p^n}(\lambda)\right)
=\sum_{i=1}^n|x_i|^{\,p}\frac {t\lambda_i}{(1+(t\lambda_i)^{\frac 1 {p-1}})^{p-1}}.$$
It was proved by Sedaev \cite{Sed} (cf. \cite{Sp2}) that if
$K_p\left(t,y^0;\, \overline{\ell_p^n}(\lambda)\right)\le K_p\left(t,x^0;\overline{\ell_p^n}(\lambda)\right)$ for all $t>0$
then there is $T:\overline{\ell_p^n}(\lambda)\to \overline{\ell_p^n}(\lambda)$ of norm
at most $2^{1/p'}$ such that $Tx^0=y^0$. (Here $p'$ is the exponent conjugate to $p$.)

Although our present estimates are particular for the case $p=2$,
our construction still shows that,
if we re-define $P(t)$ to be the polynomial
 \begin{equation}\label{pai}\frac {P(t)}{L(t)}=\sum_1^n(\tilde{x}_i^{0})^{\,p}\frac {\beta_i}{t+\beta_i}-\sum_1^n(\tilde{y}_i^{0})^{\,p}
 \frac {\alpha_i}{t+\alpha_i},\end{equation}
 then the matrix $T$ defined by
\begin{equation}\label{lp}\tau_{ik}=\re\left[\frac 1 {\alpha_i-\beta_k}\frac {(\tilde{x}_k^{0})^{p-1}}{(\tilde{y}_i^{0})^{p-1}}
\frac {\beta_k L_\delta(-\alpha_i)L_c(-\alpha_i)L_\alpha(-\beta_k)}
{\alpha_iL_\delta(-\beta_k)L_c(-\beta_k)L_\alpha'(-\alpha_i)}\right]\end{equation}
will satisfy $T\tilde{x}^{0}=\tilde{y}^{0}$, at least, provided that $P(t)>0$ when $t\ge 0$.
(Here
$L_\delta$ and $L_c$ are constructed from the zeros of $P$ as in the case $p=2$.)

The matrix \eqref{lp} differs from those used by Sedaev \cite{Sed} and
Sparr \cite{Sp2}.
Indeed the matrices from \cite{Sed,Sp2} have \textit{non-negative entries}, while this is not so for the matrices \eqref{lp}.
It seems to be an interesting problem to estimate the norm $\left\|\,T\,\right\|_{\,\bo(\overline{\ell_p^n}(\lambda))}$
for the matrix \eqref{lp}, when $p\ne 2$. The motivation for this type of question is somewhat elaborated in
\textsection\ref{ointp}, but we shall not discuss it further here.

\subsection{A comparison with L\"owner's matrix} \label{compa} In this subsection, we briefly explain
 how our matrix $T$
is related to the matrix used by L\"owner \cite{L} in his original work on monotone matrix functions.
(\footnote{By "L\"owner's matrix", we mean the unitary matrix denoted "$V$" in Donoghue's
book \cite{D0}, on p. 71. A more explicit construction of
this matrix is found in \cite{L}, where it is called "$T$".})

We shall presently display four kinds of partial isometries; L\"owner's matrix will be recognized as one of them. In all cases, operators with the required
properties can alternatively be found using the more general construction in Theorem \ref{mlem}.

The following discussion was inspired by the earlier work of Sparr \cite{Sp},
who seems to have been the first to note that L\"owner's matrix could be constructed in a similar way.

In this subsection, scalars are assumed to be real. In particular, when we write "$\ell_2^n$" we mean
the (real) Euclidean $n$-dimensional space.

Suppose that two vectors $x^{0}, y^{0}\in\R^n$ satisfy
$$\kfun_\lambda\left(t,y^{0}\right)\le \kfun_\lambda\left(t,x^{0}\right),\qquad t>0.$$
Let
$$L_\lambda(t)=\prod_1^n\left(t+\lambda_i\right),$$
and let $P\in\calP_{n-1}$ be the polynomial fulfilling
$$\frac {P(t)} {L_\lambda(t)}=\kfun_\lambda\left(t,x^{0}\right)-\kfun_\lambda\left(t,y^{0}\right)=
\sum_{i=1}^n \frac {\lambda_i}{t+\lambda_i}\left[(x_i^0)^{\,2}-(y_i^0)^{\,2}\right].$$
By assumption, $P(t)\ge 0$ for $t\ge 0$.
Moreover, $P$ is uniquely determined by the $n$ conditions
$$P(-\lambda_i)=\frac {(x_i^0)^{\,2}-(y_i^0)^{\,2}}{\lambda_iL_\lambda'(-\lambda_i)}.$$
Let $u_1,v_1,u_2,v_2,\ldots$ denote the canonical basis of $\ell_2^n$ and let
$$\ell_2^n=O\oplus E$$
be the corresponding splitting, i.e.,
$$O=\spa\,\{u_i\}\quad ,\quad E=\spa\,\{v_i\}.$$
Notice that
$$\dim O=\lfloor(n-1)/2\rfloor+1\quad,\quad \dim E=\lfloor(n-2)/2\rfloor+1,$$
where $\lfloor x\rfloor$ is the integer part of a real number $x$.

We shall construct matrices $T\in M_n(\R)$ such that
\begin{equation}\label{A1}Tx^{0}=y^{0}\quad\text{and}\quad \kfun_\lambda\left(t,Tx\right)\le
\kfun_\lambda(t,x),\quad t>0,\,x\in\mathbf{R}^n,\end{equation}
in the following special cases:
\begin{enumerate}
\item \label{case1} $P(t)=q(t)^2$ where $q\in\calP_{(n-1)/2}(\R)$, $x^{0}\in O$, and $y^{0}\in E$,
\item \label{case2} $P(t)=tq(t)^2$ where $q\in\calP_{(n-2)/2}(\R)$, $x^{0}\in E$, and $y^{0}\in O$.
\end{enumerate}
Here $\calP_x$ should be interpreted as $\calP_{\lfloor x\rfloor}$.

\begin{rem} In this connection, it is interesting to recall the well-known fact that any polynomial $P$
which is non-negative on $\mathbf{R}_+$ can be written $P(t)=q_0(t)^2+tq_1(t)^2$
for some real polynomials $q_0$ and $q_1$.
\end{rem}

To proceed with the solution, we rename the $\lambda_i$ as $\lambda_i=\xi_i$ when $i$ is odd
and $\lambda_i=\eta_i$ when $i$ is even. We also write
$$L_\xi(t)=\prod_{i\,\text{odd}}(t+\xi_i)\quad ,\quad L_\eta(t)=\prod_{i\,\text{even}}(t+\eta_i),$$
and write $L=L_\xi L_\eta.$ Notice that $L_\lambda'(-\xi_i)>0$ and $L_\lambda'(-\eta_i)<0$.

\subsubsection*{Case 1} Suppose that $P(t)=q(t)^2$, $q\in \calP_{(n-1)/2}(\R)$, $x^{0}\in O$, and $y^{0}\in E$.
Then
$$\frac {q(t)^2}{L_\lambda(t)}=\sum_{k\,\text{odd}}\frac {\xi_k} {t+\xi_k}
(x_k^0)^{\,2}-\sum_{i\, \text{even}}\frac {\eta_i} {t+\eta_i}(y_i^0)^{\,2},$$
where
\begin{equation}\label{1m}x_k^0=\frac {\eps_kq(-\xi_k)}{\sqrt{\xi_k L_\lambda'(-\xi_k)}}
\quad ,\quad y_i^0=\frac {\zeta_iq(-\eta_i)}{\sqrt{-\eta_i L_\lambda'(-\eta_i)}}
\end{equation}
for some choice of signs $\eps_k,\zeta_i\in\{\pm1\}$.

By \eqref{1m} are defined linear maps
$$O\to \calP_{(n-1)/2}(\R)\quad :\quad x\mapsto Q\quad ;\quad \calP_{(n-1)/2}(\R)\to E\quad
:\quad Q\mapsto y.$$
The composition is a linear map
$$T_0:O\to E\quad :\quad x\mapsto y.$$
We now define $T\in M_n(\R)$ by
$$T:O\oplus E\to O\oplus E\quad :\quad [x; v]\mapsto [0; T_0 x].$$
Then clearly $Tx^0=y^0$ and
\begin{equation}\label{A3}\begin{split}
\kfun_\lambda\left(t,[x; v]\right)&-\kfun_\lambda\left(t,T[x; v]\right)\\
&\ge \kfun_\xi(t,x)-\kfun_\eta\left(t,T_0x\right)\\
&=\frac {Q(t)^2} {L_\lambda(t)}\ge 0,\quad t>0,\, x\in O,\,v\in E.\\
\end{split}
\end{equation}
We have verified \eqref{A1} in case 1. A computation similar to the one in the proof
of Theorem \ref{comp} shows that, with respect to the bases $u_k$ and $v_i$,
$$(T_0)_{ik}=\frac {\eps_k\zeta_i}{\xi_k-\eta_i}\frac {L_\xi(-\eta_i)}{L_\xi'(-\xi_k)}
\left(\frac {\xi_k L_\xi'(-\xi_k)L_\eta(-\xi_k)}
{-\eta_iL_\xi(-\eta_i)L_\eta'(-\eta_i)}\right)^{1/2}.$$

Notice that, multiplying \eqref{A3} by $t$, then letting $t\to\infty$ implies that
$$\sum_{k\, \text{odd}}x_k^2\xi_k-\sum_{i\,\text{even}}(T_0x)_i^2\eta_i=0.$$
This means that $T$ \textit{is a partial isometry from $O$ to $E$ with respect to the norm
of $\ell_2^n(\lambda)$}.

\subsubsection*{Case 2.} Now assume that $P(t)=tq(t)^2$, $q\in\calP_{(n-2)/2}(\R)$, $x^0\in E$, and
$y^0\in O$. Then
$$\frac {tq(t)^2}{L_\lambda(t)}=-\sum_{i\,\text{odd}}(y_i^0)^2\frac {\xi_i}{t+\xi_i}
+\sum_{k\,\text{even}}\frac {\eta_k}{t+\eta_k}(x_k^0)^2,$$
where
\begin{equation}\label{A4}y_i^0=\frac {\eps_i'q(-\xi_i)}{\sqrt{L_\lambda'(-\xi_i)}}
\quad ,\quad x_k^0=\frac {-\zeta_k'q(-\eta_k)}{\sqrt{-L_\lambda'(-\eta_k)}}\end{equation}
for some $\eps_i',\zeta_k'\in\{\pm 1\}$.

By \eqref{A4} are defined linear maps
$$E\to\calP_{(n-2)/2}(\R)\quad :\quad x\mapsto Q\quad ;\quad \calP_{(n-2)/2}(\R)\to O\quad
:\quad Q\mapsto y.$$
We denote their composite by
$$T_1:E\to O\quad :\quad x\mapsto y.$$
Define $T\in M_n(\R)$ by
$$T:O\oplus E\to O\oplus E\quad :\quad [u; x]\mapsto \left[T_1x; 0\right].$$
We then have
\begin{equation}\label{A6}\begin{split}
-\kfun_\lambda\left(t,T[u; x]\right)&+\kfun_\lambda\left(t,[u; x]\right)\\
&\ge -\kfun_\xi\left(t,T_1x\right)+\kfun_\eta(t,x)\\
&=\frac {tQ(t)^2}{L_\lambda(t)}\ge 0,\quad t>0,\, u\in O,\, x\in E,\\
\end{split}
\end{equation}
and \eqref{A1} is verified also in case \ref{case2}.

A computation shows that, with respect to the bases $v_k$ and $u_i$,
$$(T_1)_{ik}=\frac {\eps_i'\zeta_k'}{\eta_k-\xi_i}\frac {L_\eta(-\xi_i)}
{L_\eta'(-\eta_k)}\left(
\frac {-L_\xi(-\eta_k)L'_\eta(-\eta_k)} {L_\xi'(-\xi_i)L_\eta(-\xi_i)}
\right)^{1/2}.$$
Inserting $t=0$ in \eqref{A6} we find that
$$-\sum_{i\, \text{odd}}(T_1x)_i^2+\sum_{k\,\text{even}}(x_k)^2=0,$$
i.e., $T$ \textit{is a partial isometry form $E$ to $O$ with respect to the norm
of $\ell_2^n$}.

\medskip

In the case of even $n$, the matrix $T_1$ coincides with L\"owner's matrix.

\section{Quadratic interpolation spaces}

\subsection{A classification of quadratic interpolation spaces} \label{quadro}
Recall that an intermediate space $X$ with respect to $\overline{\calH}$ is said to be
of \textit{type} $H$ if $\left\|\,T\,\right\|_{\,\bo\left(\calH_i\right)}\le M_i$ for $i=0,1$ implies that
$\left\|\,T\,\right\|_{\,\bo(X)}\le H\left(M_0,M_1\right)$. We shall henceforth make a mild restriction,
and assume that $H$ be
homogeneous of degree one. This means that we can write
\begin{equation}\label{hom1}H(s,t)^{\,2}=s^{\,2}\,\typeH(t^{\,2}/s^{\,2})\end{equation}
for some function $\typeH$ of one positive variable. In this situation, we will say that $X$ is of type $\typeH$.
The definition is chosen so that the estimates $\left\|\,T\,\right\|_{\,\bo\left(\calH_i\right)}^{\,2}\le M_i$ for $i=0,1$
imply $\left\|\,T\,\right\|_{\,\bo(X)}^{\,2}\le M_0\,\typeH\left(M_1/M_0\right)$.

In the following we will make the \textit{standing assumptions}: $\typeH$ is an increasing, continuous, and positive function on $\mathbf{R}_+$ with
$\typeH(1)=1$ and $\typeH(t)\le\max\{1,t\}$.

Notice that our assumptions imply that all spaces of type $\typeH$ are exact interpolation.
Note also that $\typeH(t)=t^{\,\theta}$ corresponds to geometric interpolation of exponent $\theta$.

Suppose now that $\overline{\calH}$ is a regular Hilbert couple and that
$\calH_*$ is an exact interpolation space
with corresponding operator $B$. By Donoghue's lemma, we have that $B=h(A)$ for some positive Borel function
$h$ on $\sigma(A)$.

The statement that $\calH_*$ is intermediate relative to
$\overline{\calH}$
is equivalent to that
\begin{equation}\label{intermed}c_1\frac A {1+A}\le B\le c_2(1+A)\end{equation}
for some positive numbers $c_1$ and $c_2$.

Let us momentarily assume that $\calH_0$ be \textit{separable}.
(This restriction is removed in Remark \ref{nonsep}.)
We can then define the \textit{scalar-valued spectral measure} of $A$,
$$\nu_A(\omega)=\sum 2^{-k}\left\langle E(\omega)e_k,e_k\right\rangle_0$$
where $E$ is the spectral measure of $A$, $\{e_k;\, k=1,2,\ldots\}$ is an orthonormal basis for $\calH_0$, and $\omega$ is a Borel set.
Then, for Borel functions $h_0,h_1$ on $\sigma(A)$, one has that
$h_1=h_2$ almost everywhere with respect to $\nu_A$ if and only if $h_1(A)=h_2(A)$.

Note that the regularity of
$\overline{\calH}$ means that $\nu_A(\{0\})=0$.

\begin{mth} \label{pcomb} If $\calH_*$ is of type $\typeH$ with respect to $\overline{\calH}$, then
$B=h(A)$ where
the function $h$
can be modified on a null-set with respect to $\nu_A$
so that
\begin{equation}\label{H-conc}h(\lambda)/h(\mu)\le \typeH\left(\lambda/\mu\right),\quad\lambda,\mu\in
\sigma(A)\setminus\{0\}.\end{equation}
\end{mth}

\begin{proof} Fix a (large) compact subset $K\subset \sigma(A)\cap\mathbf{R}_+$
and put $\calH_0'=\calH_1'=E_K(\calH_0)$ where $E$ is the spectral measure of $A$, and the norms
are defined by restriction,
$$\left\|\,x\,\right\|_{\calH_i'}=\left\|\,x\,\right\|_{\calH_i}\quad , \quad \left\|\,x\,\right\|_{\calH_*'}=\left\|\,x\,\right\|_{\calH_*}\quad ,\quad
x\in E_K\left(\calH_0\right).$$
It is clear that the operator $A'$ corresponding to $\overline{\calH'}$ is the compression
of $A$ to $\calH_0'$ and likewise the operator $B'$ corresponding to $\calH_*'$
is the compression of $B$ to $\calH_0'$. Moreover, $\calH_*'$ is of interpolation type $\typeH$ with
respect to $\overline{\calH'}$ and the operator $B'=\left(h|_K\right)(A')$.
For this reason, and since the compact set $K$
is arbitrary, it clearly suffices to prove the statement with $\overline{\calH}$ replaced by $\overline{\calH'}$.
Then $A$ is bounded above and below. Moreover, by \eqref{intermed}, also $B$ is bounded above and below.

Let $c<1$ be a positive number such that $\sigma(A)\subset\left(c,c^{-1}\right)$. For a fixed $\eps>0$
with $\eps<c/2$ we set
$$E_\lambda= \sigma(A)\cap(\lambda-\eps,\lambda+\eps)$$
and consider the functions
\begin{align*}m_\eps(\lambda)=\essinf_{E_\lambda} h,\qquad
M_\eps(\lambda)=\esssup_{E_\lambda} h,
\end{align*}
the essential inf and sup being taken with respect to $\nu_A$.

Now fix a small positive number $\eps'$
and two unit vectors
$e_\lambda, e_\mu$ supported by $E_\lambda,E_\mu$ respectively, such that
$$\left\|\, e_\lambda\,\right\|_*^{\,2}\ge M_\eps(\lambda)-\eps',\qquad \left\|\,e_\mu\,\right\|_*^{\,2}
\le m_\eps(\mu)+\eps'.$$

Now fix $\lambda,\mu\in\sigma(A)$ and let $Tx=\left\langle x,e_\mu\right\rangle_0e_\lambda$. Then
\begin{align*}\left\|\,Tx\,\right\|_1^{\,2}&=\left|\left\langle x,e_\mu\right\rangle_0\right|^{\,2}\left\|\,e_\lambda\,\right\|_1^{\,2}\le \frac 1 {(\mu-\eps)^{\,2}}\left|\left\langle x,e_\mu\right\rangle_1\right|^{\,2}(\lambda+\eps)\\
&\le
\frac {(\mu+\eps)(\lambda+\eps)} {(\mu-\eps)^2}\left\|\,x\,\right\|_1^{\,2}.
\end{align*}
Likewise,
$$\left\|\,Tx\,\right\|_0^{\,2}\le \left|\left\langle x,e_\mu\right\rangle_0\right|^{\,2}\le \left\|\,x\,\right\|_0^{\,2},$$
so $\left\|\,T\,\right\|\le 1$ and $\left\|\,T\,\right\|_A^{\,2}\le \alpha_{\mu,\lambda,\eps}$ where
$\alpha_{\mu,\lambda,\eps}=\frac {(\mu+\eps)(\lambda+\eps)}{(\mu-\eps)^2}$.

Since $\calH_*$ is of type $\typeH$,
we conclude that
$$\left\|\,T\,\right\|_B^{\,2}\le \typeH\left(\alpha_{\mu,\lambda,\eps}\right),$$ whence
\begin{equation}\label{meps}
\begin{split}
M_\eps(\lambda)-\eps'&\le \left\|\,e_\lambda\,\right\|_*^{\,2}=\left\|\,Te_\mu\,\right\|_*^{\,2}\le \typeH\left(\alpha_{\mu,\lambda,\eps}\right)\left\|\,e_\mu\,\right\|_*^{\,2}\\
&\le \typeH\left(\alpha_{\mu,\lambda,\eps}\right)\left(m_\eps(\mu)+\eps'\right).\\
\end{split}
\end{equation}
In particular, since $\eps'$ was arbitrary, and $m_\eps(\lambda)\le\left\|\,e_\lambda\,\right\|_*^{\,2}\le \left\|\,B\,\right\|$, we find that
$$M_\eps(\lambda)-m_\eps(\lambda)\le \left[\typeH\left(\alpha_{\mu,\lambda,\eps}\right)-1\right]
\left\|\,B\,\right\|.$$
 By assumption, $\typeH$ is continuous and $\typeH(1)=1$. Hence, as $\eps\downarrow 0$, the functions $M_\eps(\lambda)$ diminish
monotonically, converging uniformly to a function $h_*(\lambda)$ which is also the uniform limit of the
family $m_\eps(\lambda)$. It is clear that $h_*$ is continuous, and since $m_\eps\le h_*\le M_\eps$,
we have $h_*=h$ almost everywhere with respect to $\nu_A$. The relation \eqref{H-conc} now follows
for $h=h_*$ by letting $\eps$ and $\eps'$ tend to zero in \eqref{meps}.
\end{proof}

A partial converse to Theorem \ref{pcomb} is found below, see Theorem \ref{fth1}.

\begin{rem} \label{nonsep} (The non-separable case.) Now consider the case when $\calH_0$ is non-separable.
(By regularity this means that also $\calH_1$ and $\calH_*$ are non-separable.)

First assume that the operator $A$ is bounded.
Let $\calH_0'$ be a separable reducing subspace for $A$ such that the restriction $A'$ of $A$ to $\calH_0'$
has the same spectrum as $A$. The space $\calH_0'$ reduces $B$ by Donoghue's lemma; by Theorem \ref{pcomb}
the restriction $B'$ of $B$ to $\calH_0'$ satisfies $B'=h'(A')$ for some continuous function $h'$
satisfying \eqref{H-conc} on
$\sigma(A)$.
Let $\calH_0''$ be any other separable reducing subspace, where (as before) $B''=h''(A'')$.
Then $\calH_0'\oplus\calH_0''$ is a separable reducing subspace on which $B=h(A)$ for some
third continuous function $h$ on $\sigma(A)$. Then $h(A')\oplus h(A'')=h'(A')\oplus h''(A'')$
and by continuity we must have $h=h'=h''$ on $\sigma(A)$. The
function $h$ thus satisfies $B=h(A)$ as well as the estimate \eqref{H-conc}.

If $A$ is unbounded, we replace $A$ by its compression to $P_n\calH_0$ where $P_n$ is the
spectral projection of $A$ corresponding to the spectral set $[0,n]\cap\sigma(A)$, $n=1,2,\ldots$. The same reasoning as above
shows that $B$ appears as a continuous function of $A$ on $\sigma(A)\cap[0,n]$. Since $n$ is arbitrary,
we find that $B=h(A)$ for a function $h$ satisfying \eqref{H-conc}.
\end{rem}

\subsection{Geometric interpolation}
Now consider the particular case when $\calH_*$ is of exponent $\theta$, viz. of type $\typeH(t)=t^{\,\theta}$
with respect to $\overline{\calH}$. We write $B=h(A)$ where $h$ is the continuous function
provided by Theorem \ref{pcomb} (and Remark \ref{nonsep} in the non-separable case).

Fix a point $\lambda_0\in\sigma(A)$ and let $C=h(\lambda_0)\lambda_0^{\,-\theta}$.
The estimate \eqref{H-conc} then implies that $h(\lambda)\le C\lambda^\theta$
and $h(\mu)\ge C\mu^\theta$ for all $\lambda,\mu\in\sigma(A)$. We have proved
the following theorem.

\begin{thm} \label{pthm} (\cite{Mc,U})
If $\calH_*$ is an exact interpolation Hilbert space
of exponent $\theta$ relative to $\overline{\calH}$, then $B=h(A)$ where
$h(\lambda)=C\lambda^{\,\theta}$ for some
positive constant $C$.
\end{thm}

Theorem \ref{pthm} says that $\calH_*=\calH_\theta$ up to a constant multiple of the norm, where
$\calH_\theta$ is the space defined in
\eqref{mc}.
In the guise of operator inequalities:
for any fixed positive operators $A$ and $B$, the condition
$${T}^*T\le M_0\quad ,\quad {T}^*AT\le M_1A\quad \Rightarrow\quad {T}^*BT\le M_0^{\,\,1-\theta}M_1^{\,\,\theta} B$$
is equivalent to that $B=A^{\,\theta}$.

It was observed in \cite{Mc} that $\calH_\theta$ also equals to the complex interpolation space
$C_\theta(\,\overline{\calH}\,)$. For the sake of completeness, we supply a short proof of this
fact in the appendix.

\begin{rem} \label{urem} An exact quadratic interpolation method, the \textit{geometric mean} was introduced earlier by Pusz and Woronowicz \cite{PW} (it corresponds to the $C_{1/2}$-method).
In \cite{U}, Uhlmann generalized that method to a method (the \textit{quadratic mean}) denoted $\QI_t$ where $0<t<1$; this method is
quadratic and of exponent $t$.

In view of Theorem \ref{pthm} and the preceding remarks we can conclude that $\QI_\theta(\,\overline{\calH}\,)=C_\theta(\,\overline{\calH}\,)=
\calH_\theta$ for any regular Hilbert couple $\overline{\calH}$.
We refer to \cite{U} for several physically relevant applications of this type of interpolation.

Finally, we want to mention that in \cite{P4} Peetre introduces
the "Riesz method of interpolation"; in Section 5
he also defines a related method "$\QM$" which comes close to the
complex $C_{1/2}$-method.
\end{rem}

\subsection{Donoghue's theorem} The exact quadratic
interpolation spaces relative to a Hilbert couple were characterized by Donoghue in the paper \cite{D1}.
We shall here prove the following equivalent version of Donoghue's result (see \cite{A2,A3}).

\begin{mth} \label{dthm} An intermediate Hilbert space $\calH_*$ relative to $\overline{\calH}$ is an exact interpolation space
if and only if there is a positive radon measure $\vro$ on $[0,\infty]$ such that
$$\|\,x\,\|_*^{\,2}=\int_{[0,\infty]}\left(1+t^{-1}\right)K\left(t,x;\,\overline{\calH}\,\right)\, d\vro(t).$$
Equivalently, $\calH_*$ is exact interpolation relative to $\overline{\calH}$ if and only
if the corresponding operator $B$ can be represented as $B=h(A)$ for some function
$h\in P'$.
\end{mth}

The statements that all norms of the given form are exact quadratic interpolation norms have already been
shown (see \textsection\ref{mex}). There remains to prove that there are no others.

Donoghue's original
 formulation of the result, as well as other equivalent forms of the theorem, is found in Section \ref{repif}
below. Our present approach follows \cite{A2} and is based on
$K$-monotonicity.

\begin{rem}\label{qcrem}
The condition that $\calH_*$ be exact interpolation with respect to
$\overline{\calH}$ means that $\calH_*$ is of type $\typeH$ where $\typeH(t)=\max\{1,t\}$.
In view of Theorem \ref{pcomb} (and Remark \ref{nonsep}), this means that we can represent $B=h(A)$ where
$h$ is \textit{quasi-concave} on $\sigma(A)\setminus\{0\}$,
\begin{equation}\label{quasic}h(\lambda)\le h(\mu)\max\left\{1,\lambda/\mu\right\},\qquad \lambda,\mu\in\sigma(A)\setminus\{0\}.\end{equation}
In particular, $h$ is locally Lipschitzian on $\sigma(A)\cap \R_+$.
\end{rem}

\begin{rem} A related result concerning \textit{non-exact} quadratic interpolation was proved by
Ovchinnikov \cite{O} using Donoghue's theorem. Cf. also \cite{A5}.
\end{rem}

\subsection{The proof for simple finite-dimensional couples} \label{fslem} Similar to
our approach to Calderón's problem, our strategy is to reduce Theorem \ref{dthm} to a case of "simple couples''.

\begin{thm} \label{dono1} Assume that $\calH_0=\calH_0=\C^{\,n}$ as sets
and that
all eigenvalues $(\lambda_i)_1^n$ of the corresponding operator $A$ are of unit multiplicity.
Consider a third
 Hermitian norm $\|x\|_*^{\,2}=\left\langle Bx,x\right\rangle_0$ on $\C^{\,n}$. Then $\calH_*$ is exact interpolation
 with respect to $\overline{\calH}$ if and only if $B=h(A)$ where
$h\in P'$.
\end{thm}

\begin{rem}
The lemma says that the class of functions $h$ on $\sigma(A)$ satisfying
\begin{equation}\label{solve}{T}^*T\le 1\quad,\quad {T}^*AT\le A\quad \Rightarrow\quad {T}^*h(A)T\le h(A), \qquad (T\in M_n(\C))\end{equation}
is precisely the set $P'|\sigma(A)$ of restrictions of $P'$-functions to $\sigma(A)$.
In this way, the condition \eqref{solve} provides an operator-theoretic solution
to the interpolation problem by positive Pick functions on a finite subset of $\mathbf{R}_+$.
\end{rem}

\begin{proof}[Proof of Theorem \ref{dono1}] We already know that the spaces $\calH_*$ of the asserted
form are exact interpolation relative to $\overline{\calH}$ (see subsections \ref{mex} and \ref{opint}).

Now let $\calH_*$ be any exact quadratic interpolation space.
By Donoghue's lemma and the argument
in \textsection \ref{tml}, we can for an appropriate positive sequence $\lambda=(\lambda_i)_1^n$
identify
$\overline{\calH}=\overline{\ell_2^n}(\lambda)$,
$A=\diag(\lambda_i)$, and $B=h(A)$ where $h$ is some positive function defined on $\sigma(A)=\{\lambda_i\}_1^n$.

Our assumption is that $\ell_2^n\left(h(\lambda)\right)$ is exact interpolation relative to
$\overline{\ell_2^n}(\lambda)$. We must prove that $h\in P'|\sigma(A)$.
To this end, write
$$k_{\lambda_i}(t)=\frac {(1+t)\lambda_i} {1+t\lambda_i},$$
and recall that (see Lemma \ref{kcalc})
$$K\left(t,x;\overline{\ell_2^n}(\lambda)\right)=\left(1+t^{-1}\right)^{-1}\sum_1^n|x_i|^{\,2}\,k_{\lambda_i}(t).$$
Let us denote by $C$ the algebra of continuous complex functions on $[0,\infty]$ with the supremum norm
$\left\|\,u\,\right\|_\infty=\sup_{t>0}|\,u(t)\,|$. Let $V\subset C$ be the linear span of the $k_{\lambda_i}$
for $i=1,\ldots,n$. We define a positive functional $\phi$ on $V$ by
$$\phi(\sum_1^n a_ik_{\lambda_i})=\sum_1^n a_i\,h(\lambda_i).$$
We claim that $\phi$ is a \textit{positive functional}, i.e., if $u\in V$ and $u(t)\ge 0$ for all $t>0$,
then $\phi(u)\ge 0$.

To prove this let $u=\sum_1^n a_ik_{\lambda_i}$ be non-negative on $\mathbf{R}_+$ and write
$a_i=|x_i|^{\,2}-|y_i|^{\,2}$ for some $x,y\in\C^n$. The condition that $u\ge 0$ means that
\begin{equation}\label{chur}\begin{split}\left(1+t^{-1}\right)K\left(t,x;\overline{\ell_2^n}(\lambda)\right)&=
\sum_{i=1}^n|x_i|^{\,2}\,k_{\lambda_i}(t)\\
&\ge \sum_{i=1}^n|y_i|^{\,2}\,k_{\lambda_i}(t)\\
&=\left(1+t^{-1}\right)K\left(t,y;\overline{\ell_2^n}(\lambda)\right),\qquad t>0.\\
\end{split}
\end{equation}
Since $\overline{\ell_2^n}(\lambda)$ is an exact Calderón couple (by Theorem \ref{mlem}), the space $\ell_2^n(h(\lambda))$ is exact $K$-monotonic. In other words,
\eqref{chur} implies that
$$\left\|\,x\,\right\|_{\ell_2^n(h(\lambda))}\ge \left\|\,y\,\right\|_{\ell_2^n(h(\lambda))},$$
i.e.,
$$\phi(u)=\sum_1^n\left(|x_i|^{\,2}-|y_i|^{\,2}\right)\,h(\lambda_i)\ge 0.$$
The asserted positivity of $\phi$ is thereby proved.

Replacing $\lambda_i$ by $c\lambda_i$ for a suitable positive constant $c$ we can without losing
generality assume that $1\in\sigma(A)$, i.e., that the unit $\1(x)\equiv 1$ of the $C^*$-algebra $C$ belongs to $V$.
The positivity of $\phi$ then ensures that
$$\left\|\,\phi\,\right\|=\sup_{u\in V;\, \|u\|_\infty\le 1}\left|\phi(u)\right|=\phi(\1).$$
Let $\Phi$ be a Hahn-Banach extension of $\phi$ to $C$ and note that
$$\left\|\,\Phi\,\right\|=\left\|\,\phi\,\right\|=\phi(\1)=\Phi(\1).$$
This means that $\Phi$ is a positive functional on $C$ (cf. \cite{Mu}, \textsection 3.3). By the Riesz representation
theorem there is thus a positive Radon measure $\vro$ on $[0,\infty]$ such that
$$\Phi(u)=\int_{[0,\infty]}u(t)\, d\vro(t),\qquad u\in C.$$
In particular
$$h(\lambda_i)=\phi\left(k_{\lambda_i}\right)=\Phi\left(k_{\lambda_i}\right)=\int_{[0,\infty]}
\frac {(1+t)\lambda_i} {1+t\lambda_i}\, d\vro(t),\quad i=1,\ldots,n.$$
We have shown that $h$ is the restriction to $\sigma(A)$ of a function of class $P'$.
\end{proof}

\subsection{The proof of Donoghue's theorem} We here prove Theorem \ref{dthm} in full generality.

We remind the reader that if $S\subset\R_+$ is a subset, we write $P'|S$ for the convex cone of restrictions of $P'$-functions
to $S$. We first collect some simple facts about this cone.

\begin{lem} \label{flem} (i) The class $P'|S$ is closed under pointwise convergence.

(ii) If $S$ is finite and if $\lambda=(\lambda_i)_{i=1}^n$ is an enumeration of the points
of $S$ then $h$ belongs to $P'|S$ if and only if $\ell_2^n(h(\lambda))$ is exact interpolation
with respect to the pair $\overline{\ell_2^n}(\lambda)$.

(iii) If $S$ is infinite, then a continuous
function $h$ on $S$ belongs to $P'|S$ if and only if $h\in P'|\Lambda$ for every finite subset $\Lambda\subset
S$.
\end{lem}

\begin{proof}
(i) Let $h_n$ be a sequence in $P'$
converging pointwise on $S$ and fix $\lambda\in S$. It is clear that the boundedness of the numbers $h_n(\lambda)$
is equivalent to boundedness of the total masses of the corresponding measures $\vro_n$ on the compact set $[0,\infty]$.
It now suffices to apply Helly's selection theorem.

(ii) This is Theorem \ref{dono1}.

(iii) Let $\Lambda_n$ be an increasing sequence of finite subsets of $S$
whose union is dense. Let $h_n=h|\Lambda_n$ where $h$ is continuous on $S$.
If $h_n\in P'|\Lambda_n$ for all $n$ then the sequence $h_n$ converges pointwise on $\cup\Lambda_n$ to
$h$. By part (i) we then have $h\in P'|\sigma(A)$.
\end{proof}

We can now finish the proof of Donoghue's theorem (Theorem \ref{dthm}).

Let $\calH_*$ be exact interpolation with respect to $\overline{\calH}$ and represent the
corresponding operator as $B=h(A)$ where
$h$ satisfies \eqref{H-conc}.
By the remarks after Theorem \ref{dthm}, the function $h$ is locally Lipschitzian.

In view of Lemma \ref{flem} we shall be done when we have proved that
$\ell_2^n(h(\lambda))$ is exact interpolation with respect to $\overline{\ell_2^n}(\lambda)$
for all sequences $\lambda=(\lambda_i)_1^n\subset\sigma(A)$ of distinct points. Let us arrange the sequences in
increasing order: $0<\lambda_1<\cdots<\lambda_n$.

Fix $\eps>0$, $\eps<\min\{c,\lambda_1,1/\lambda_n\}$ and let $E_i=[\lambda_i-\eps,\lambda_i+\eps]\cap\sigma(A)$; we assume that
$\eps$ is sufficiently small that the $E_i$ be disjoint. Let $M=\cup_1^n E_i$. We can assume that $h$ has
Lipschitz constant at most $1$ on $M$.

Let $\calM$ be the reducing subspace of $\calH_0$ corresponding to the spectral set $M$, and let
$\tilde{A}$ be the compression of $A$ to $\calM$. We define a function $g$ on $M$ by
$g(\lambda)=\lambda_i$ on $E_i$. Then $\left|g(\lambda)-\lambda\right|<\eps$ on $\sigma(\tilde{A})$, so
\begin{equation}\|\,\tilde{A}-g(\tilde{A})\,\|\le\eps\quad ,\quad
\|\,h(\tilde{A})-h(g(\tilde{A}))\,\|\le \eps.\end{equation}

\begin{lem} \label{yneq} Suppose that $A',A''\in \bo\left(\calM\right)$ satisfy $A',A''\ge \delta>0$ and $\left\|\,A'-A''\,\right\|\le\eps$.
Then $\left\|\,T\,\right\|_{A''}\le\sqrt{1+2\eps/\delta}\,
\max\{\left\|\,T\,\right\|,
\left\|\,T\,\right\|_{A'}\}$ for all $T\in \bo\left(\calM\right)$.
\end{lem}

\begin{proof} By definition, $\left\|\,T\,\right\|_{A'}$ is the smallest number $C\ge 0$ such that
${T}^*A'T\le C^{\,2}A'.$ Thus
\begin{align*}
{T}^*A''T&= T^*(A''-A')T+T^*A'T\\
&\le \left\|\,T\,\right\|^{\,2}\eps+\left\|\,T\,\right\|_{A'}^{\,2}\left(A''+\left(A'-A''\right)\right)\\
&\le 2\eps\max\{\left\|\,T\,\right\|^{\,2},\left\|\,T\,\right\|_{A'}^{\,2}\}+
\left\|\,T\,\right\|_{A'}^{\,2}A''\\
&\le\max\{\left\|\,T\,\right\|^{\,2},\left\|\,T\,\right\|_{A'}^{\,2}\}\left(1+2\eps/\delta\right)A''.
\end{align*}
\end{proof}

We can find $\delta>0$ such that the operators $\tilde{A}$, $g(\tilde{A})$, $h(\tilde{A})$, and
$h(g(\tilde{A}))$ are $\ge \delta$. Then by repeated use of Lemma \ref{yneq},
\begin{align*}\left\|\,T\,\right\|_{h\left(g\left(\tilde{A}\right)\right)}&\le \sqrt{1+2\eps/\delta}\,\max\{\left\|\,T\,\right\|,\left\|\,T\,\right\|_{h\left(\tilde{A}\right)}\}\\
&\le \sqrt{1+2\eps/\delta}\,\max\left\{\left\|\,T\,\right\|,\left\|\,T\,\right\|_{\tilde{A}}\right\}\\
&\le \left(1+2\eps/\delta\right)\max\{\left\|\,T\,\right\|,\left\|\,T\,\right\|_{g\left(\tilde{A}\right)}\},\quad T\in \bo(\calM).
\end{align*}
Let $e_i$ be a unit vector supported by the spectral set $E_i$ and define a space $\calV\subset\calM$
to be the $n$-dimensional space spanned by the $e_i$. Let $A_0$ be the compression of
$g(\tilde{A})$ to $\calV$; then
\begin{equation}\label{LAST}\left\|\,T\,\right\|_{h\left(A_0\right)}\le \left(1+2\eps/\delta\right)\max\left\{\left\|\,T\,\right\|,\left\|\,T\,\right\|_{A_0}\right\},\quad T\in \bo\left(\calV\right).\end{equation}
Identifying $\calV$ with $\ell_2^n$ and $A_0$ with the matrix $\diag(\lambda_i)$, we see that \eqref{LAST} is independent of
$\eps$. Letting $\eps$ diminish to $0$ in \eqref{LAST} now gives that $\ell_2^n(h(\lambda))$
is exact interpolation with respect to $\overline{\ell_2^n}(\lambda)$. In view of Lemma \ref{flem},
this finishes the proof of Theorem \ref{dthm}. q.e.d.

\section{Classes of matrix functions}

In this section, we discuss the basic properties of interpolation functions: in particular, the relation to the well known classes of monotone matrix functions.
We refer to the books \cite{D0} and \cite{RR} for further reading on the latter classes.

\subsection{Interpolation and matrix monotone functions} \label{monofun}
Let $A_1$ and $A_2$ be positive
operators in $\ell_2^n$ ($n=\infty$ is admitted). Suppose that $A_1\le A_2$ and
form the following operators on $\ell_2^n\oplus\ell_2^n$,
$$T_0=\begin{pmatrix}0 & 0 \cr 1 & 0\cr\end{pmatrix}\quad ,\quad
A=\begin{pmatrix} A_2 & 0 \cr 0 & A_1\cr \end{pmatrix}.$$
It is then easy to see that ${T_0}^*T_0\le 1$ and that ${T_0}^*AT_0=\begin{pmatrix} A_1 & 0\cr
0 & 0\cr \end{pmatrix} \le A$.

Now assume that
a function $h$ on $\sigma(A)$ belongs to the
class $C_A$ defined in \textsection \ref{donolemm}, i.e., that $h$ satisfies
\begin{equation}\label{giverise}T^*T\le 1\quad ,\quad T^*AT\le A\quad \Rightarrow\quad T^*h(A)T\le h(A),\end{equation}
where $T$ denotes an operator on $\ell_2^{2n}$.

We then have ${T_0}^*h(A)T_0\le h(A)$, or
$$\begin{pmatrix} h(A_1) & 0\cr 0 & 0\cr \end{pmatrix} \le \begin{pmatrix} h(A_2) & 0\cr
0 &  h(A_1)\cr \end{pmatrix}.$$
In particular, we find that
$h(A_1)\le h(A_2)$. We have shown that (under the assumptions
above)
\begin{equation}\label{mono}A_1\le A_2\qquad \Rightarrow\qquad h\left(A_1\right)\le h\left(A_2\right).\end{equation}

We now change our point of view slightly. Given a positive integer $n$, we let
$C_n$ denote the convex of positive functions $h$ on $\R_+$ such that
\eqref{giverise} holds for \textit{all}
positive operators $A$ on $\ell_2^n$ and all $T\in \bo\left(\ell_2^n\right)$.

Similarly, we let
$P_n'$ denote the class of all positive functions $h$ on $\R_+$ such that
$h(A_1)\le h(A_2)$ whenever $A_1,A_2$ are positive operators on $\ell_2^n$ such that $A_1\le A_2$. We refer to $P_n'$ as the cone of
positive functions \textit{monotone of order $n$} on $\R_+$.

We have shown above that $C_{2n}\subset P_{n}'$.

In the other direction, assume that $h\in P_{2n}'$. Let $A, T$ be bounded operators on $\ell_2^n$ with
$A>0$, $T^*T\le 1$ and $T^*AT\le A$. Assume also that $h$ be continuous.
We will use the following lemma due to Hansen \cite{H0}. We recall the proof for completeness.

\begin{lem}\label{hans} (\cite{H0}) $T^*h(A)T\le h(T^*AT)$.
\end{lem}

\begin{proof} Put $S=(1-TT^*)^{1/2}$ and $R=(1-T^*T)^{1/2}$ and consider the $2n\times 2n$ matrix
$$U=\begin{pmatrix}T & S\cr
R & -T^*\cr
\end{pmatrix}\qquad,\qquad X=\begin{pmatrix} A & 0\cr
0 & 0\cr\end{pmatrix}.$$
It is well-known, and easy to check, that $U$ is unitary and that
$$U^*XU=\begin{pmatrix} T^*AT & T^*AS\cr SAT & SAS\cr\end{pmatrix}.$$

Next fix a number $\eps>0$, a constant $\lambda>0$ (to be fixed), and form the matrix
$$Y=\begin{pmatrix}T^*AT+\eps & 0\cr 0 & 2\lambda\cr\end{pmatrix}$$
which, provided that we choose $\lambda\ge \|SAS\|$, satisfies
$$Y-U^*XU=\begin{pmatrix}\eps & -T^*AS\cr -SAT & 2\lambda-SAS\end{pmatrix}\ge \begin{pmatrix}\eps & D\cr
D^* & \lambda\end{pmatrix},$$
where we have written $D=-T^*AS$.

If we now also choose $\lambda$ so that
$\lambda\ge \|D\|^2/\eps$, then
we obtain for all $\xi,\eta\in\C^n$ that
\begin{align*}\left\langle \begin{pmatrix}\eps & D\cr
D^* & \lambda\end{pmatrix}\begin{pmatrix}\xi\cr\eta\cr\end{pmatrix}\,,\,\begin{pmatrix}\xi\cr\eta\cr\end{pmatrix}\right\rangle&=
\eps\|\xi\|^2+\langle D\eta,\xi\rangle+\langle D^*\xi,\eta\rangle+\lambda\|\eta\|^2\\
&\ge \eps\|\xi\|^2-2\|D\|\|\xi\|\|\eta\|+\lambda\|\eta\|^2\ge 0.
\end{align*}
Hence $U^*XU\le Y$ and as a consequence $U^*h(X)U=h(U^*XU)\le h(Y)$, since $h$ is matrix monotone of order $2n$.
The last inequality means that
$$\begin{pmatrix}T^*h(A)T & T^*h(A)S\cr
Sh(A)T & Sh(A)S\cr\end{pmatrix}\le \begin{pmatrix} h(T^*AT+\eps) & 0\cr
0 & h(2\lambda)\cr
\end{pmatrix},$$
so in particular $T^*h(A)T\le h(T^*AT+\eps)$. Since $\eps>0$ was arbitrary, and since $h$ is assumed to be continuous,
we conclude the lemma.
\end{proof}

We now continue our discussion. Assuming that $T^*T\le 1$ and $T^*AT\le A$, and that $h\in P_{2n}'$ is continuous, we have $h(T^*AT)\le h(A)$ [since $h\in P_n'$],
so $T^*h(A)T\le h(A)$ by Lemma \ref{hans}. We conclude that $h\in C_n$.

To prove that $P_{2n}'\subset C_n$, we need to remove the continuity assumption on $h$ made above.
This is completely standard: let $\fii$ be a smooth positive function on $\R_+$ such that
$\int_0^\infty \fii(t)\, dt/t=1$, and define a sequence $h_k$ by $h_k(\lambda)=k^{-1}\int_0^\infty
\fii\left(\lambda^k/t^k\right)h(t)\, dt/t$. The class $P_{2n}'$ is a convex cone, closed under pointwise
convergence \cite{D0}, so the functions $h_1,h_2,\ldots$ are of class $P_{2n}'$. They are furthermore
continuous, so by the argument above, they are of class $C_n$. By Lemma \ref{flem}, the cone
$C_n$ is also closed under pointwise convergence, so $h=\lim h_n\in C_n$.

To summarize, we have the inclusions $C_{2n}\subset P_n'$, $P_{2n}'\subset C_n$, and also $C_{n+1}\subset C_n$, $P_{n+1}'\subset P_n'$.
In view of Theorem \ref{dthm}, we have the identity $\cap_1^\infty C_n=P'$.
The inclusions above now imply the following result, sometimes known as "Löwner's theorem on matrix monotone functions".

\begin{thm} \label{loewn} We have $\cap_1^\infty P_n'=\cap_1^\infty C_n=P'$.
\end{thm}

The identity $\cap_1^\infty P_n'=P'$  says that a positive function $h$ is monotone of all finite orders if
and only it is of class $P'$. The somewhat less precise fact that
$P_\infty'=P'$ is interpreted as that the class of \textit{operator monotone functions} coincides with $P'$.

The identity $C_\infty=P'$ is, except for notation, contained in the work of
Foia\c{s} and Lions, from \cite{FL}. See \textsection \ref{jfl}.

Note that the inclusion $P_{2n}'\subset C_n$ shows that a matrix monotone functions of order $2n$ can be interpolated by a positive Pick function at $n$ points.
Results of a similar nature, where it is shown, in addition, that an interpolating Pick function can be taken rational of a certain degree, are discussed, for example, in Donoghue's book \cite[Chapter XIII]{D1} or (more relevant in the present connection) in the paper \cite{D2}.

It seems somewhat inaccurate to refer to the identity $\cap_1^\infty P_n'=P'$ as "L\"owner's theorem", since
L\"owner discusses more subtle results concerning matrix monotone functions \textit{of a given finite order} $n$.
In spite of this, it is common nowadays to let "L\"owner's theorem" refer to this identity.

\subsection{More on the cone $C_A$} We can now give an short proof of the following result due to Donoghue \cite{D2}.

\begin{thm} \label{adprop} For a positive function $h$ on $\sigma(A)$ we define
two positive functions $\tilde{h}$ and $h^*$ on
$\sigma\left(A^{-1}\right)$ by $\tilde{h}(\lambda)=\lambda
h\left(1/\lambda\right)$ and $h^*(\lambda)=1/h\left(1/\lambda\right)$.
Then the following conditions are equivalent,
\begin{enumerate}
\item[(i)] $h\in C_A$,
\item[(ii)] $\tilde{h}\in C_{A^{-1}}$,
\item[(iii)] $h^*\in C_{A^{-1}}$.
\end{enumerate}
\end{thm}

\begin{proof} Let $\calH_*$ be a quadratic intermediate
space relative to a regular Hilbert couple $\overline{\calH}$;
let $B=h(A)$ be the corresponding operator. It is clear that $\calH_*$
is exact interpolation relative to $\overline{\calH}$ if and only
if $\calH_*$ is exact interpolation relative to the \textit{reverse couple}
$\overline{\calH^{(r)}}=\left(\calH_1,\calH_0\right)$.
The latter couple has corresponding operator $A^{-1}$ and
it is clear that the identity $\|\,x\,\|_*^{\,2}=\langle h(A)x,x\rangle_0$
is equivalent to that $\|\,x\,\|_*^{\,2}=\left\langle A^{-1}\tilde{h}\left(A^{-1}\right)x,x\right\rangle_1$.
We have shown the equivalence of (i) and (ii).

Next let $\overline{{\calH}^*}=\left({\calH_0}^*,{\calH_1}^*\right)$ be the
\textit{dual couple}, where we identify ${\calH_0}^*=\calH_0$.
With this identification, ${\calH_1}^*$ becomes associated with the
norm $\|\,x\,\|_{\calH_1^{\,*}}^{\,2}=\left\langle A^{-1}x,x\right\rangle_0$, and
${\calH_*}^{\,*}$ is associated with $\|\,x\,\|_{\calH_*^{\,*}}^{\,2}=
\left\langle B^{-1}x,x\right\rangle_0$. It remains to note that $\calH_*$ is exact
interpolation relative to $\overline{\calH}$ if and only if
${\calH_*}^{\,*}$ is exact interpolation relative to
$\overline{\calH^*}$, proving the equivalence of (i) and (iii).
\end{proof}

Combining with Theorem \ref{dthm}, one obtains alternative
proofs of the interpolation theorems for $P'$-functions discussed by
Donoghue in the paper \cite{D2}.

\begin{rem} The exact quadratic interpolation spaces
which are fixed by the duality, i.e., which satisfy ${\calH_*}^{\,*}=\calH_*$,
correspond precisely to the class of $P'$-functions which are \textit{self-dual}: $h^*=h$. This class was characterized by Hansen
in the paper \cite{H1}.
\end{rem}

\subsection{Matrix concavity} A function $h$ on $\R_+$ is called \textit{matrix
concave of order $n$} if we have Jensen's inequality
$$\lambda h\left(A_1\right)+(1-\lambda)h\left(A_2\right)\le h\left(\lambda A_1+(1-\lambda)A_2\right)$$
for all positive $n\times n$ matrices $A_1$, $A_2$, and all numbers $\lambda\in [0,1]$. Let
us denote by $\Gamma_n$ the convex cone of positive concave functions of order $n$ on $\R_+$.
The fact that $\cap_n\Gamma_n=P'$ follows from the theorem of Kraus \cite{Kr}. Following \cite{A2} we now give
an alternative proof of this fact.

\begin{prop} For all $n$ we have the inclusion $C_{3n}\subset \Gamma_n\subset P_n'$. In particular
$\cap_1^\infty \Gamma_n=P'$.
\end{prop}

\begin{proof} Assume first that $h\in C_{3n}$ and pick two positive matrices $A_1$ and $A_2$.
Define $A_3=(1-\lambda)A_1+\lambda A_2$ where $\lambda\in[0,1]$ is given, and define
matrices $A$ and $T$ of order $3n$ by
$$A=\begin{pmatrix}A_3 & 0 & 0 \cr 0 & A_1 & 0\cr 0 & 0 & A_2\cr\end{pmatrix}\qquad
,\qquad T=\begin{pmatrix} 0 & 0 & 0 \cr
\sqrt{1-\lambda} & 0 & 0\cr \sqrt{\lambda} & 0 & 0\cr \end{pmatrix}.$$
It is clear that $T^*T\le 1$ and
$$T^*AT=\begin{pmatrix}A_3 & 0 & 0 \cr 0 & 0 & 0\cr 0 & 0 & 0\cr\end{pmatrix}\le A,$$
so, since $h\in C_{3n}$, we have $T^*h(A)T\le h(A)$, or
$$\begin{pmatrix}(1-\lambda)h(A_1)+\lambda h(A_2) & 0 & 0 \cr 0 & 0 & 0\cr 0 & 0 & 0\cr\end{pmatrix}\le
\begin{pmatrix}h(A_3) & 0 & 0 \cr 0 & h(A_1) & 0\cr 0 & 0 & h(A_2)\cr\end{pmatrix}.$$
Comparing the matrices in the upper left corners, we find that $h\in \Gamma_n$.

Assume now that $h\in \Gamma_n$, and take positive definite matrices $A_1,A_2$ of order $n$ with
$A_1\le A_2$. Also pick $\lambda\in (0,1)$. Then $\lambda A_2=\lambda A_1+(1-\lambda)A_3$ where
$A_3=\lambda(1-\lambda)^{-1}(A_2-A_1)$. By matrix concavity, we then have
$$h(\lambda A_2)\ge \lambda h(A_1)+(1-\lambda)h(A_3)\ge \lambda h(A_1),$$
where we used non-negativity to deduce the last inequality. Being concave, $h$ is certainly
continuous.
Letting $\lambda\uparrow 1$ one thus
finds that $h(A_1)\le h(A_2)$. We have shown that $h\in P_n'$.
\end{proof}

For a further discussion of classes of convex matrix functions and their relations to monotonicity, we refer to the paper \cite{He}.

\subsection{Interpolation functions of two variables} In this section, we briefly discuss a class of
interpolation functions of two matrix variables.
We shall not completely characterize the class of such generalized interpolation functions here, but we hope that the following discussion
will be of some use for a future investigation.

\smallskip

Let $H_1$ and $H_2$ be Hilbert spaces. We turn $H_1\otimes H_2$
into a Hilbert space by defining the inner product on elementary
tensors via $\left\langle x_1\otimes x_2,x_1'\otimes x_2'\right\rangle:=
\left\langle x_1,{x_1}'\right\rangle_1\cdot\left\langle x_2,{x_2}'\right\rangle_2$ (then extend
via sesqui-linearity). Similarly,
if $T_i$ are operators on $H_i$, the tensor product $T_1\otimes T_2$
is defined on elementary tensors via $\left(T_1\otimes T_2\right)(x_1\otimes x_2)=
T_1x_1\otimes T_2x_2$. It is then easy to see that if
$A_i$ are positive operators on $H_i$ for $i=1,2$, then
$A_1\otimes A_2\ge 0$ as an operator on the tensor product.
Furthermore, we have $A_1\otimes A_2\le A_1'\otimes A_2'$ if
$A_i\le A_i'$ for $i=1,2$.

\smallskip

Given two positive definite matrices $A_i$ of orders
$n_i$ and a function $h$ on $\sigma(A_1)\times\sigma(A_2)$, we define a matrix $h(A_1,A_2)$ by
$$h(A_1,A_2)=\sum_{(\lambda_1,\lambda_2)\in\sigma(A_1)\times\sigma(A_2)}h\left(\lambda_1,\lambda_2\right)E^1_{\lambda_1}\otimes E^2_{\lambda_2}$$
where $E^j$ is the spectral resolution of the matrix $A_j$.

We shall say that $h$ \textit{gives rise to exact interpolation} relative to $(A_1,A_2)$, and
write $h\in C_{A_1,A_2}$, if the condition
\begin{equation}\label{pork}{T_j}^*T_j\le 1\qquad ,\qquad {T_j}^*A_jT_j\le A_j,\qquad j=1,2\end{equation}
implies
\begin{equation}\begin{split}\label{zork}h(A_1,A_2)&+(T_1\otimes T_2)^*h(A_1,A_2)(T_1\otimes T_2)\\
&-(T_1\otimes 1)^*h(A_1,A_2)(T_1\otimes 1)-(1\otimes T_2)^*h(A_1,A_2)(1\otimes T_2)\ge 0.\\
\end{split}
\end{equation}
Taking $T_1=T_2=0$ we see that $h\ge 0$ for all $h\in C_{A_1,A_2}$. It is also clear that
$C_{A_1,A_2}$ is a convex cone closed under pointwise convergence on the finite set
$\sigma(A_1)\times\sigma(A_2)$.

If $h=h_1\otimes h_2$ is an elementary tensor
where $h_j\in C_{A_j}$ is a function
of one variable, then \eqref{pork} implies ${T_j}^*h_j(A_j)T_j\le h_j(A_j)$, whence $(h_1(A_1)-{T_1}^*h_1(A_1)T_1)\otimes (h_2(A_2)-{T_2}^*h_2(A_2)T_2)\ge 0$, which implies \eqref{zork}. We have shown that $C_{A_1}\otimes C_{A_2}\subset C_{A_1,A_2}$.

Since for each $t\ge 0$ the $P'$-function $\lambda\mapsto \frac {(1+t)\lambda} {1+t\lambda}$ is of class $C_{A_j}$,
we infer that every function representable in the form
\begin{equation}\label{p2}h(\lambda_1,\lambda_2)=\iint_{[0,\infty]^{\,2}}\frac {(1+t_1)\lambda_1} {1+t_1\lambda_1}
\frac {(1+t_2)\lambda_2} {1+t_2\lambda_2}\, d\vro(t_1,t_2)\end{equation}
with some positive Radon measure $\vro$ on $[0,\infty]^{\, 2}$ is in the class $C_{A_1,A_2}$.

We shall say that a function $h$ on $\sigma(A_1)\times \sigma(A_2)$ has the \textit{separate interpolation-property}
if for each fixed $b\in \sigma(A_2)$ the function $\lambda_1\mapsto h(\lambda_1,b)$ is of class $C_{A_1}$,
and a similar statement holds for all functions $\lambda_2\mapsto h(a,\lambda_2)$.

\begin{lem} \label{munch} Each function of class $C_{A_1,A_2}$ has the separate interpolation-property.
\end{lem}

\begin{proof} Let $T_2=0$ and take an arbitrary $T_1$ with ${T_1}^*T_1\le 1$ and ${T_1}^*A_1T_1\le A_1$.
By hypothesis,
$$(T_1\otimes 1)^*h(A_1,A_2)(T_1\otimes 1)\le
h(A_1,A_2).$$
Fix an eigenvalue $b$ of $A_2$ and let $y$ be a corresponding normalized eigenvector.
Then
for all $x\in H_1$ we have $\left\langle h(A_1,A_2)x\otimes y,x\otimes y\right\rangle=\left\langle h(A_1,b)x,x\right\rangle_{H_1}$ and
$\left\langle (T_1\otimes 1)^*h(A_1,A_2)(T_1\otimes 1)x\otimes y,x\otimes y\right\rangle=\left\langle {T_1}^*h(A_1,b)T_1x,x\right\rangle_{H_1}$
so
$$\left\langle {T_1}^*h(A_1,b)T_1x,x\right\rangle_{H_1}\le\left\langle h(A_1,b)x,x\right\rangle_{H_1}.$$
The functions $h(a,\lambda_2)$ can be treated similarly.
\end{proof}

\begin{ex} The function $h(\lambda_1,\lambda_2)=(\lambda_1+\lambda_2)^{1/2}$ clearly has the separate
interpolation-property for all $A_1,A_2$. However, it is not representable in the form \eqref{p2}.
Indeed, $\re\{h(i,i)-h(-i,i)\}=1$ while it is easy to check that
$\re\{h(\lambda_1,\lambda_2)-h(\bar{\lambda}_1,\lambda_2)\}\le 0$ whenever $\im\lambda_1,\im\lambda_2>0$ and
$h$ is of the form \eqref{p2}.
\end{ex}

Let us say that a function $h(\lambda_1,\lambda_2)$ defined on $\R_+\times\R_+$ is an
interpolation function (of two variables) if
$h\in C_{A_1,A_2}$ for all $A_1,A_2$.
Lemma \ref{munch}
implies that interpolation functions are separately real-analytic in $\R_+\times\R_+$ and that the functions $h(a,\cdot)$ and $h(\cdot,b)$ are of class $P'$ (cf. Theorem \ref{dthm}).

The above notion of interpolation function is close to Korányi's definition of monotone
matrix function of two variables: $f(\lambda_1,\lambda_2)$ is \textit{matrix monotone} in a rectangle
$I=I_1\times I_2$ ($I_1$, $I_2$ intervals in $\R$) if
$A_1\le A_1'$ (with spectra in $I_1$) and $A_2\le A_2'$ (with spectra in $I_2$) implies
$$f(A_1',A_2')-f(A_1',A_2)-f(A_1,A_2')+f(A_1,A_2)\ge 0.$$

\begin{lem} \label{im} Each interpolation function is matrix monotone in $\R_+\times\R_+$.
\end{lem}

\begin{proof} Let $0<A_i\le A_i'$ and put
$\tilde{A}_i=\begin{pmatrix}A_i' & 0\cr 0 & A_i\cr\end{pmatrix},\, T_i=\begin{pmatrix} 0 & 0\cr 1 & 0\cr
\end{pmatrix}.$
Since ${T_i}^*\tilde{A}_iT_i\le \tilde{A}_i$, an interpolation function $h$ will satisfy
the interpolation inequality \eqref{zork} with $A_i$ replaced by $\tilde{A_i}$. Applying this inequality to
vectors of the form $\begin{pmatrix}x_1\cr 0\cr\end{pmatrix}\otimes\begin{pmatrix}x_2\cr 0\cr\end{pmatrix}$ we readily obtain
\begin{align*}&\left\langle h(A_1',A_2')x_1\otimes x_2,x_1\otimes x_2\right\rangle-\left\langle h(A_1,A_2')x_1\otimes x_2,x_1\otimes x_2\right\rangle\\
&-
\left\langle h(A_1',A_2)x_1\otimes x_2,x_1\otimes x_2\right\rangle+\left\langle h(A_1,A_2)x_1\otimes x_2,x_1\otimes x_2\right\rangle\ge 0.
\end{align*}
The same result obtains with $x_1\otimes x_2$ replaced by a sum $x_1\otimes x_2+x_1'\otimes x_2'+\ldots$,
i.e., $h$ is matrix monotone.
\end{proof}

\begin{rem} Assume that $f$ is of the form $f(\lambda_1,\lambda_2)=g_1(\lambda_1)+g_2(\lambda_2)$.
Then $f$ is matrix monotone for all $g_1$, $g_2$ and $f$ is an
interpolation function if and only if $g_1,g_2\in P'$. In order to disregard "trivial" monotone functions
of the above type, Korányi \cite{K} imposed the normalizing assumption (a) $f(\lambda_1,0)=f(0,\lambda_2)=0$
for all $\lambda_1,\lambda_2$.
\end{rem}

It follows from Lemma \ref{im} and the proof of \cite[Theorem 4]{K} that, if $h$ is a $C^2$-smooth interpolation
function, then the function
$$k(x_1,x_2;y_1,y_2)=\frac {h(x_1,x_2)-h(x_1,y_2)-h(y_1,x_2)+h(y_1,y_2)}{(x_1-y_1)(x_2-y_2)}$$
is \textit{positive definite} in the sense that $\sum_m\sum_n k(x_m,y_m;x_n,y_n)\alpha_m\bar{\alpha}_n\ge 0$
for all finite sequences of positive numbers $x_j,y_k$ and all complex numbers $\alpha_l$.
(The proof uses Löwner's matrix.) Korányi uses
essentially this positive definiteness condition (and condition (a) in the remark above)
to deduce an integral representation formula
for $h$ as an integral of products of Pick functions. See Theorem 3 in \cite{K}. However, in contrast to our situation, Korányi
considers functions monotone on the rectangle $(-1,1)\times(-1,1)$, so this last result cannot be
immediately applied. (It easily implies local representation formulas, valid in finite rectangles, but
these representations do not appear to be very natural from our point of view.)

This is not the right place to attempt to extend Korányi's methods to functions on $\R_+\times\R_+$; it would
seem more appropriate to give a more direct characterization of the classes $C_{A_1,A_2}$ or of the class of interpolation functions. At present, we do
not know if there is an interpolation function which is not representable in the form \eqref{p2}.

\section{Proof of the $K$-property} \label{appa}
In this section we extend the result of Theorem \ref{mlem} to obtain the full proof
of Theorem \ref{mthm}. The discussion is in principle not hard, but it does require some care to keep track of both norms when reducing to a finite-dimensional case.

Recall first that, by Lemma \ref{diagred},
it suffices to consider the diagonal case $\overline{\calH}=\overline{\calK}$.

To prove Theorem \ref{mthm} we fix a regular Hilbert couple $\overline{\calH}$; we must prove that it has the \Kpr
(see \textsection \ref{capa}). By Theorem \ref{mlem}, we know that this is true if $\overline{\calH}$ is finite dimensional and the
associated operator only has eigenvalues of unit multiplicity.

We shall use a weak* type compactness result (\cite{A2}).
To formulate it, let
$\bo_1(\,\overline{\calH}\,)$ be the unit ball in the space $\bo(\,\overline{\calH}\,)$.
Moreover, let $\Sigma_t$ be
the sum $\calH_0+\calH_1$ normed by $\left\|\,x\,\right\|_{\Sigma_{\,t}}^{\, 2}:=K(t,x)$. Note that $\|\cdot\|_{\Sigma_t}$ is an equivalent
norm on $\Sigma$ and that
$\Sigma_{\,1}=\Sigma$ isometrically.
We denote by
$\bo_1\left(\,\Sigma_{\,t}\,\right)$ the unit ball in the space $\bo\left(\,\Sigma_{\,t}\,\right)$.

In view of Remark \ref{simrem}, one has the identity
\begin{equation}\label{thecomp}\bo_1(\,\overline{\calH}\,)=
\bigcap_{t\in\R_+}\bo_1\left(\,\Sigma_{\,t}\,\right).\end{equation}
We shall use this to define a compact topology on $\bo_1(\,\overline{\calH}\,)$.

\begin{lem} \label{complem} The subset $\bo_1\left(\,\overline{\calH}\,\right)\subset \bo_1\left(\,\Sigma\,\right)$ is compact relative to the
weak operator topology inherited from $\bo\left(\,\Sigma\,\right)$.
\end{lem}

Recall that the \textit{weak operator topology on $\bo\left(\,H\,\right)$} is the weakest topology such that a net
$T_i$ converges to the limit $T$ if
the inner product $\left\langle T_ix,y\right\rangle_H$ converges to $\left\langle Tx,y\right\rangle_H$ for all $x,y\in H$.

\begin{proof}[Proof of Lemma \ref{complem}] The weak operator topology coincides on the unit ball $\bo_1\left(\,\Sigma\,\right)$ with
the weak*-topology, which
is compact, due to Alaoglu's theorem (see \cite{Mu}, Chap. 4 for details). It is clear that for a fixed $t>0$, the subset
$\bo_1\left(\,\Sigma\,\right)\cap \bo_1\left(\,\Sigma_{\,t}\,\right)$ is weak operator closed in $\bo_1\left(\,\Sigma\,\right)$; hence
it is also compact. In view of \eqref{thecomp}, the set $\bo_1\left(\,\overline{\calH}\,\right)$ is an intersection
of compact sets. Hence the set $\bo_1\left(\,\overline{\calH}\,\right)$ is itself compact, provided that we endow it with the subspace topology inherited
from $\bo_1\left(\,\Sigma\,\right)$.
\end{proof}

Denote by $P_n$ the projections $P_n=E_{\sigma(A)\cap\left[n^{-1},n\right]}$ on $\calH_0$
where $E$ is the spectral resolution of $A$ and $n=1,2,3,\ldots$. Consider the couple
$$\overline{\calH^{(n)}}=(P_n\left(\calH_0\right),P_n\left(\calH_1\right)),$$
the associated operator of which is the compression $A_n$ of $A$ to the subspace $P_n\left(\calH_0\right)$.
Note that the norms in the couple $\overline{\calH^{(n)}}$ are equivalent, i.e.,
the associated operator $A_n$ is bounded above and below.

We shall need two lemmas.

\begin{lem} \label{wll} If $\overline{\calH^{(n)}}$ has the \Kpr for all $n$, then so
does $\overline{\calH}$.
\end{lem}

\begin{proof} Note that $\left\|\,P_n\,\right\|_{\,\bo(\overline{\calH})}=1$ for all $n$, and
that $P_n\to 1$ as $n\to \infty$ relative to the strong operator topology on
$\bo\left(\Sigma\right)$. Suppose that
$x^0,y^0\in\Sigma$ are elements such that, for some $\rho>1$,
\begin{equation}\label{bashyp}K\left(t,y^0\right)< \frac 1 \rho\, K\left(t,x^0\right),\qquad t>0.\end{equation}
Then
$K\left(t,P_ny^0\right)\le K\left(t,y^0\right)<\rho^{-1}K\left(t,x^0\right)$. Moreover,
the identity $K\left(t,P_ny^0\right)=\left\langle \frac {tA_n}{1+tA_n}P_ny^0,P_ny^0\right\rangle_0$
shows that we have an estimate of the form $K(t,P_ny^0)\le C_n\min\{1,t\}$
for $t>0$ and large enough $C_n$ (this follows since $A_n$ is bounded above and below).

The functions $K\left(t,P_mx^0\right)$ increase monotonically, converging
uniformly on compact subsets of $\mathbf{R}_+$ to $K\left(t,x^0\right)$
when $m\to\infty$. By concavity of the function $t\mapsto K\left(t,P_mx^0\right)$
we will then have
\begin{equation}\label{bycc}K\left(t,P_ny^0\right)< \frac 1 {\tilde{\rho}}K\left(t,P_mx^0\right),
\qquad t\in\R_+,\end{equation}
provided that $m$ is sufficiently large, where $\tilde{\rho}$ is
any number in the interval $1<\tilde{\rho}<\rho$.

Indeed, let $A=\lim_{t\to\infty} K\left(t,P_ny^0\right)$ and $B=\lim_{t\to 0}K\left(t,P_ny^0\right)/t$.
Take points $t_0<t_1$ such that $K(t,P_ny^0)\ge A/\rho'$ when $t\ge t_1$ and
$K(t,P_ny^0)/t\le B\rho'$ when $t\le t_0$. Here $\rho'$ is some number in the interval $1<\rho'<\rho$.

Next use \eqref{bashyp} to choose $m$ large enough that
$K(t,P_mx^0)>\rho K(t,P_ny^0)$ for all $t\in[t_0,t_1]$. Then
$K(t,P_mx^0)>(\rho/\rho')K(t,P_ny^0)$ for $t=t_1$, hence for all $t\ge t_1$, and
$K(t,P_mx^0)/t>(\rho/\rho')K(t,P_ny^0)/t$ for $t=t_0$ and hence also when $t\le t_0$.
Choosing $\rho'=\rho/\tilde{\rho}$ now
establishes \eqref{bycc}.

Put $N=\max\{m,n\}$. If $\overline{\calH^{(N)}}$ has the \Kpr,
we can find a map
$T_{nm}\in \bo_1(\,\overline{\calH}\,)$ such that
$T_{nm}P_mx^0=P_ny^0$. (Define $T_{mn}=0$ on the orthogonal complement of $P_N\left(\calH_0\right)$ in $\Sigma$.)
In view of Lemma \ref{complem}, the maps
$T_{nm}$ must cluster at some point $T\in \bo_1(\,\overline{\calH}\,)$.
It is clear that $Tx^0=y^0$. Since $\rho>1$ was arbitrary, we have shown that $\overline{\calH}$ has the \Kpr.
\end{proof}

\begin{lem} \label{52} Given $x^0,y^0\in \calH^{(n)}_0$ and a number $\epsilon>0$
there exists a positive integer $n$ and a finite-dimensional couple $\overline{\calV}\subset \overline{\calH^{(n)}}$ such that
$x^0,y^0\in\calV_0+\calV_1$ and
\begin{equation}(1-\epsilon)K\left(t,x;\overline{\calH}\right)\le
K\left(t,x;\overline{\calV}\right)\le (1+\epsilon)K\left(t,x;\overline{\calH}\right),
\quad t>0,\, x\in\calV_0+\calV_1.\end{equation}
Moreover, $\overline{\calV}$ can be chosen so that all eigenvalues of the associated
operator $A_{\overline{\calV}}$ are of unit multiplicity.
\end{lem}

\begin{proof} Let $A_n$ be the operator associated with the couple $\overline{\calH^{(n)}}$;
thus $1/n\le A_n\le n$.

Take $\eta>0$ and let $\{\lambda_i\}_1^N$ be a finite subset
of $\sigma\left(A_n\right)$ such that $\sigma\left(A_n\right)\subset
\cup_1^N E_i$ where $E_i=(\lambda_i-\eta/2,\lambda_i+\eta/2)$.
We define a Borel function $w:\sigma\left(A_n\right)\to\sigma\left(A_n\right)$ by
$w(\lambda)=\lambda_i$ on $E_i\cap\sigma\left(A_n\right)$; then
$\left\|\,w\left(A_n\right)-A_n\,\right\|_{\,\bo\left(\calH_0\right)}\le\eta$.

Let $k_t(\lambda)=\frac {t\lambda}{1+t\lambda}$. It is easy to check
that the Lipschitz constant of the restriction $k_t\bigm|\sigma\left(A_n\right)$ is bounded above by
$C_1\min\{1,t\}$ where $C_1=C_1(n)$ is independent of $t$. Hence
$$\left\|\,k_t\left(w\left(A_n\right)\right)-k_t\left(A_n\right)\,\right\|_{\,\bo\left(\calH_0\right)}\le C_1\eta\min\left\{1,t\right\}.$$
It follows readily that
\begin{align*}
\left|\left\langle\left(k_t\left(w\left(A_n\right)\right)-k_t\left(A_n\right)\right)x,x\right\rangle_0\right|\le C_1\eta\min\{1,t\}\left\|\,x\,\right\|_0^{\,2},
\quad x\in P_n\left(\calH_0\right).\end{align*}
Now let $c>0$ be such that $A\ge c$. The elementary inequality $k_t(c)\ge (1/2)\min\{1,ct\}$ shows that
\begin{align*}\left\langle k_t\left(A_n\right)x,x\right\rangle_0
\ge C_2\min\{1,t\}\left\|\,x\,\right\|_0^{\,2},\quad x\in P_n\left(\calH_0\right),
\end{align*}
where $C_2=(1/2)\min\{1,c\}$.
Combining these estimates, we deduce that
\begin{equation}\label{54}\left|\left\langle k_t\left(w\left(A_n\right)\right)x,x\right\rangle_0-\left\langle k_t\left(A_n\right)x,x\right\rangle_0\right|
\le C_3\eta\left\langle k_t\left(A_n\right)x,x\right\rangle_0,\quad x\in P_n\left(\calH_0\right)\end{equation}
for some suitable constant $C_3=C_3(n)$.

Now pick unit vectors $e_i,f_i$ supported by the spectral sets $E_i\cap\sigma(A_n)$
such that $x^0$ and $y^0$ belong to the space $\calW$ spanned by
$\{e_i,f_i\}_1^N$. Put $\calW_0=\calW_1=\calW$ and define norms on
those spaces by
$$\left\|\,x\,\right\|_{\calW_0}=\left\|\,x\,\right\|_{\calH_0}\quad ,\quad
\left\|\,x\,\right\|_{\calW_1}^{\,2}=\left\langle w\left(A\right)x,x\right\rangle_{\calH_0}.$$
The operator associated with $\overline{\calW}$ is then the compression
of $w(A_n)$ to $\calW_0$, i.e.,
\begin{equation*}\left\|\,x\,\right\|_{\calW_1}^{\,2}=\left\langle A_{\overline{\calW}}\,x,x\right\rangle_{\calW_0}=
\left\langle w(A_n)x,x\right\rangle_{\calH_0},\qquad x\in\calW.\end{equation*}
Let $\epsilon=2C_3\eta$ and observe that, by \eqref{54}
\begin{equation}\label{56}
\left|\,K\left(t,x;\overline{\calW}\right)-K\left(t,x;\overline{\calH}\right)\,\right|\le
(\epsilon/2)K\left(t,x;\overline{\calH}\right),\quad
f\in\calW.\end{equation}
The eigenvalues of $A_{\overline{\calW}}$ typically have multiplicity
$2$. To obtain unit multiplicity, we perturb $A_{\overline{\calW}}$
slightly to a positive matrix $A_{\overline{\calV}}$ such that
$\left\|\,A_{\overline{\calW}}-A_{\overline{\calV}}\,\right\|_{\,\bo\left(\calH_0\right)}<\epsilon/2C_3$.
Let $\overline{\calV}$ be the couple associated with
$A_{\overline{\calV}}$, i.e., put $\calV_i=\calW$ for $i=0,1$ and
$$\left\|\,x\,\right\|_{\calV_0}=\left\|\,x\,\right\|_{\calW_0}\quad
\text{and}\quad \left\|\,x\,\right\|_{\calW_1}^{\,2}=\left\langle A_{\overline{\calV}}\,x,x\right\rangle_{\calV_0}.$$
It is then straightforward to check that
$$\left|\,K\left(t,f;\overline{\calW}\right)-K\left(t,f;\overline{\calV}\right)\,\right|\le
(\epsilon/2)K\left(t,f;\overline{\calH}\right),\quad f\in\calW.$$
Combining this with the estimate \eqref{56}, one finishes the proof of the lemma.
\end{proof}

\begin{proof}[Proof of Theorem \ref{mthm}] Given two elements $x^0,y^0\in \Sigma$ as in \eqref{bashyp}
we write $x^n=P_n\left(x^0\right)$ and $y^n=P_n\left(y^0\right)$. By the proof of Lemma \ref{wll}
we then have $K\left(t,y^n\right)\le \tilde{\rho}^{-1}K\left(t,x^n\right)$ for large enough $n$,
where $\tilde{\rho}$ is any given number in the interval $(1,\rho)$.

We then use Lemma \ref{52} to choose a finite-dimensional sub-couple
$\overline{\calV}\subset \overline{\calH^{(n)}}$
such that
\begin{align*}K\left(t,y^n;\overline{\calV}\right)&\le (1+\epsilon)K\left(t,y^n;\overline{\calH}\right)\\
&<\tilde{\rho}^{-1}K\left(t,x^n;\overline{\calV}\right)+\epsilon\left(K\left(t,x^n;\overline{\calH}\right)+
K\left(t,y^n;\overline{\calH}\right)\right).\\
\end{align*}
Here $\epsilon>0$ is at our disposal.

Choosing $\epsilon$ sufficiently small, we can arrange that
\begin{equation}\label{pertcon}K\left(t,y^n;\overline{\calV}\right)\le K(t,x^n;\overline{\calV}),\quad t>0.\end{equation}

By Theorem \ref{mlem}, the condition \eqref{pertcon} implies the existence
of an operator $T'\in \bo_1\left(\,\overline{\calV}\,\right)$ such that
$T'x^n=y^n$. Considering the canonical inclusion and projection
$$I:\Sigma\left(\calV\right)\to\Sigma\left(\calH\right)\quad \text{and}\quad
\Pi:\Sigma\left(\calH\right)\to\Sigma\left(\calV\right),$$
we have, by virtue of Lemma \ref{52},
$$\left\|\,I\,\right\|_{\,\bo(\overline{\calV};\overline{\calH})}^{\,2}\le (1-\epsilon)^{-1}
\quad \text{and}\quad \left\|\,\Pi\,\right\|_{\,\bo(\overline{\calH};\overline{\calV})}^{\,2}\le 1+\epsilon.$$
Now let $T=T_\eps:=IT'\Pi\in \bo(\,\overline{\calH^{(n)}}\,)$. Then
$\left\|\,T\,\right\|^{\,2}\le \frac {1+\epsilon}{1-\epsilon}$ and $Tx^n=y^n$.
As $\epsilon\downarrow 0$ the operators $T_\epsilon$ will cluster at some
point $T\in \bo_1(\,\overline{\calH^{(n)}}\,)$ such that $Tx^n=y^n$ (cf. Lemma
\ref{complem}).

We have shown that $\overline{\calH^{(n)}}$ has the \Kpr. In view of
Lemma \ref{wll}, this implies that $\overline{\calH}$ has the same property.
The proof of Theorem \ref{mthm} is therefore complete.
\end{proof}

\section{Representations of interpolation functions} \label{repif}

\subsection{Quadratic interpolation methods} Let us say that an interpolation
method defined on regular Hilbert couples taking values in Hilbert spaces is a \textit{quadratic
interpolation method}. (Donoghue \cite{D1} used the same phrase in a somewhat wider sense, allowing
the methods to be defined on non-regular Hilbert couples as well.)

If $F$ is an exact quadratic interpolation method, and $\overline{\calH}$ a Hilbert couple, then by
Donoghue's theorem \ref{dthm} there exists a positive Radon measure $\vro$ on $[0,\infty]$ such that
$F\left(\,\overline{\calH}\,\right)=\calH_\vro$, where the latter space is defined
by the familiar norm $\|x\|_\vro^{\,2}=\int_{[0,\infty]}\left(1+t^{-1}\right)K(t,x)\,d\vro(t)$.

A priori, the measure $\vro$ could depend not only on $F$ but also on the particular $\overline{\calH}$.
That $\vro$ is independent of $\overline{\calH}$ can be realized in the following way.
Let $\overline{\calH'}$ be a regular Hilbert couple such that every positive rational number is
an eigenvalue of the associated operator. Let $B'$ be the operator associated with the exact quadratic
interpolation space $F\left(\,\overline{\calH'}\,\right)$.
There is then clearly a unique $P'$-function $h$
on $\sigma\left(A'\right)$ such that $B'=h\left(A'\right)$, viz. there is a unique positive Radon measure $\vro$ on $[0,\infty]$
such that
$F\left(\,\overline{\calH'}\,\right)=\calH_\rho'$ (see \textsection \ref{mex} for the notation).

If $\overline{\calH}$ is any regular Hilbert couple, we can form the direct sum
$\overline{\calS}=\overline{\calH'}\oplus\overline{\calH}$. Denote by $\tilde{A}$ the corresponding
operator and
 let $\tilde{B}=\tilde{h}(\tilde{A})$ be the operator corresponding to the exact quadratic
 interpolation space $F\left(\,\overline{\calS}\,\right)$. Then $\tilde{h}(\tilde{A})
 =\tilde{h}(A')\oplus\tilde{h}\left(A\right)=h\left(A'\right)\oplus \tilde{h}(A)$.
 This means that $\tilde{h}\left(A'\right)=h\left(A'\right)$, i.e. $\tilde{h}=h$. In particular, the
 operator $B$ corresponding to the exact interpolation space
 $F(\,\overline{\calH}\,)$ is equal to $h\left(A\right)$. We have shown that
 $F(\,\overline{\calH}\,)=\calH_\vro$. We emphasize our conclusion with the following theorem.

 \begin{thm} \label{flthm} There is a one-to-one correspondence $\vro\mapsto F$
 between positive Radon measures and exact quadratic interpolation methods.
 \end{thm}

We will shortly see that Theorem \ref{flthm} is equivalent to the theorem of Foia\c{s} and Lions \cite{FL}.
As we remarked above, a more general version of the theorem, admitting for non-regular Hilbert couples,
is found in Donoghue's paper \cite{D1}.

\subsection{Interpolation type and reiteration} \label{fansect}
In this subsection, we prove some general facts concerning quadratic interpolation methods; we shall mostly follow Fan \cite{F}.

Fix a function $h\in P'$ of the form
$$h(\lambda)=\int_{[0,\infty]}\frac {(1+t)\lambda}
{1+t\lambda}\, d\vro(t).$$
It will be convenient to write
$\overline{\calH}_h$ for the
corresponding exact interpolation space
$\calH_\vro$.
Thus, we shall denote
$$\|x\|_h^{\,2}=\left\langle h(A)x,x\right\rangle_0=\int_{[0,\infty]}
\left(1+t^{-1}\right)K\left(t,x\right)\,d\vro(t).$$
More generally, we shall use the same notation
when $h$ is any quasi-concave
function on $\R_+$; then $\overline{\calH}_h$ is a quadratic interpolation space, but not necessarily exact.

Recall that, given a function $\typeH$ of one variable, we say
that $\calH_*$ is of type $\typeH$ with respect to
$\overline{\calH}$ if $\left\|\,T\,\right\|_{\,\bo\left(\calH_i\right)}^{\,2}\le M_i$ implies
$\left\|\,T\,\right\|_{\,\bo\left(\calH_*\right)}^{\,2}\le M_0\,\typeH\left(M_1/M_0\right)$.

We shall say that a quasi-concave function $h$ on $\R_+$ is
\textit{of type $\typeH$} if
$\overline{\calH}_h$ is of type $\typeH$ relative to
any regular Hilbert couple $\overline{\calH}$.
The following result somewhat generalizes Theorem \ref{pthm}. The class of
functions of type $\typeH$ clearly forms a convex cone.

\begin{thm} \label{macgen} Let $h$ be of type $\typeH$, where
(i) $\typeH(1)=1$ and $\typeH(t)\le\max\{1,t\}$,
and (ii)
$\typeH$ has left and right derivatives $\theta_{\pm}=H'(1\pm)$
at the point $1$, where $\theta_-\le \theta_+$. Then for any positive constant $c$,
\begin{equation}\label{quasip}\min\left\{\lambda^{\theta_-},
\lambda^{\theta_+}\right\}\le
\frac {h(c\lambda)} {h(c)}\le \max\left\{\lambda^{\theta_-},\lambda^{\theta_+}\right\},
\qquad \lambda\in\R_+.\end{equation}
In particular, if $\typeH(t)$ is differentiable at $t=1$ and $\typeH'(1)=\theta$, then
$h(\lambda)=\lambda^{\,\theta}$, $\lambda\in\R_+$.
\end{thm}

\begin{proof} Replacing $A$ by $cA$, it is easy to see that if $h$ is of type $\typeH$, then so is $h_c(t)=h(ct)/h(c)$.
Fix $\mu>0$ and
consider the function $h_0(t)=h_c(\mu t)/h_c(\mu)$.
By
Theorem \ref{pcomb}, we have $h_0(t)\le \typeH(t)$ for all $t$.
Furthermore $h_0(1)=\typeH(1)=1$ by (i). Since $h_0$ is differentiable,
the assumption (ii) now gives $\theta_-\le h_0'(1)\le \theta_+$,
or
$$\theta_-\le\frac {\mu h_c'(\mu)} {h_c(\mu)}\le \theta_+.$$
Dividing through by $\mu$ and integrating over the interval
$[1,\lambda]$, one now verifies the inequalities in \eqref{quasip}.
\end{proof}

The following result provides a partial converse to Theorem \ref{pcomb}.

\begin{thm} (\cite{F}) \label{fth1}
Let $h\in P'$ and set
$\typeH(t)=\sup_{s>0} h(st)/h(s)$. Then $h$ is of
type $\typeH$.
\end{thm}

\begin{proof} Let $T\in \bo(\overline{\calH})$ be a
non-zero operator; put $M_j=\left\|\,T\,\right\|_{\, \bo\left(\calH_j\right)}^{\,2}$
and $M=M_1/M_0$. We then have (by Lemma \ref{kcalc})
\begin{align*}\|\,Tx\,\|_h^{\,2}&=\int_{[0,\infty]}\left(1+t^{-1}\right)
K\left(t,Tx\right)\, d\vro(t)\\
&\le M_0\int_{[0,\infty]}\left(1+t^{-1}\right)
K\left(tM,x\right)\, d\vro(t)\\
&=M_0\int_{[0,\infty]}\left\langle \frac {(1+t)MA}
{1+tMA}x,x\right\rangle_0\, d\vro(t)\\
&=M_0\left\langle h\left(MA\right)x,x\right\rangle_0.
\end{align*}

Letting $E$ be the spectral resolution of $A$, we have
$$\left\langle h\left(MA\right)x,x\right\rangle_0=
\int_0^\infty h\left(M\lambda\right)\, d\left\langle E_\lambda x,x\right\rangle_0.$$
Since $h\left(M\lambda\right)/h(\lambda)\le \typeH\left(M\right)$, we conclude that
$$\left\|\, Tx\,\right\|_h^{\,2}\le M_0\typeH\left(M\right)
\int_0^\infty h\left(\lambda\right)\, d\left\langle E_\lambda x,x\right\rangle_0=M_0\typeH\left(M\right)
\left\|\,x\,\right\|_h^{\,2},$$
which finishes the proof.
\end{proof}

Given a function $h$ of a positive variable, we define a new function
$\tilde{h}$ by
$$\tilde{h}(s,t)=s\, h\left(t/s\right).$$
The following reiteration theorem is due to Fan.

\begin{thm} (\cite{F}) Let $h, h_0,h_1\in P'$, and
$\fii(\lambda)=\tilde{h}\left(h_0(\lambda),h_1(\lambda)\right)$.
Then $\overline{\calH}_\fii=(\overline{\calH}_{h_0},\overline{\calH}_{h_1}
)_h$ with equal norms. Moreover, $\overline{\calH}_\fii$ is
an exact interpolation space relative to $\overline{\calH}$.
\end{thm}

\begin{proof} Let $\overline{\calH'}$ denote the couple $(\overline{\calH}_{h_0},\overline{\calH}_{h_1}
)$. The corresponding operator $A'$ then obeys
$$\|\,x\,\|_{\overline{\calH}_{h_1}}=\|\, (A')^{1/2}x\,\|_{\calH_0'}=
\|\, \fii_0(A)^{1/2}(A')^{1/2}x\,\|_0,\quad x\in
\Delta(\,\overline{\calH'}\,).$$
On the other hand, $\|\,x\,\|_{\overline{\calH}_{h_1}}=\left\|\, \fii_1(A)^{1/2}x\,\right\|_0$, so
$$(A')^{1/2}x=\fii_0(A)^{-1/2}\fii_1(A)^{1/2}x,\quad x\in
\Delta\left(\,\overline{\calH'}\,\right).$$
We have shown that $A'=\fii_0(A)^{-1}\fii_1(A)$, whence (by Lemma \ref{kcalc})
\begin{equation}\label{kca2}\begin{split}K\left(t,x;\overline{\calH'}\right)&=
\left\langle \frac {t\fii_0(A)^{-1}\fii_1(A)}
{1+t\fii_0(A)^{-1}\fii_1(A)}x,x\right\rangle_{\calH_0'}\\
&=\left\langle \frac {t\fii_1(A)}
{1+t\fii_0(A)^{-1}\fii_1(A)}x,x\right\rangle_{\calH_0'}.\\
\end{split}
\end{equation}

Now let the function $h\in P'$ be given by
$$h(\lambda)=\int_{[0,\infty]}\frac {(1+t)\lambda} {1+t\lambda}\,
d\vro(t),$$
and note that the function $\fii=\tilde{h}\left(h_0,h_1\right)$
is given by
$$\fii(\lambda)=\int_{[0,\infty]}\frac {(1+t)h_1(\lambda)}
{1+th_1(\lambda)/h_0(\lambda)}\,d\vro(t).$$
Combining with \eqref{kca2}, we find that
\begin{align*}\|\,x\,\|_{\overline{\calH'}_h}^{\,2}&=
\int_{[0,\infty]}\left(1+t^{-1}\right)K\left(t,x;\overline{\calH'}\right)
\, d\vro(t)\\
&=\int_0^\infty\left[ \int_{[0,\infty]}\frac {(1+t)h_1(\lambda)}
{1+th_1(\lambda)/h_0(\lambda)}\, d\vro(t)\right]\,
d\left\langle E_\lambda x,x\right\rangle_0=\|\,x\,\|_{\overline{\calH}_\fii}^{\,2}.\end{align*}
This finishes the proof of the theorem.
\end{proof}

 Combining with Donoghue's theorem \ref{dthm}, one obtains the following, purely function-theoretic corollary. Curiously, we are not aware of a proof which does not use interpolation theory.

\begin{cor} (\cite{F}) Suppose that $h\in P'$ and that $h_0,h_1\in P'|F$,
where $F$ is some closed subset of $\R_+$.
Then the function $\fii=\tilde{h}(h_0,h_1)$ is also of class $P'|F$.
\end{cor}

\subsection{Donoghue's representation} Let $\overline{\calH}$ be a regular Hilbert couple. In Donoghue's
setting, the principal object is the space $\Delta=\calH_0\cap\calH_1$ normed
by $\|\,x\,\|_\Delta^{\,2}=\|\,x\,\|_0^{\,2}+\|\,x\,\|_1^{\,2}$.
In the following, all involutions are understood to be taken with respect to the norm of $\Delta$.

We express the norms in the spaces $\calH_i$ as
$$\|\,x\,\|_0^{\,2}=\left\langle Hx,x\right\rangle_\Delta\qquad \text{and}\qquad \|\,x\,\|_1^{\,2}=\left\langle (1-H)x,x\right\rangle_\Delta,$$
where $H$ is a bounded positive operator on $\Delta$, $0\le H\le 1$. The regularity of $\overline{\calH}$
means that neither $0$, nor $1$ is an eigenvalue of $H$.

To an arbitrary quadratic intermediate space $\calH_*$ there corresponds a bounded positive injective operator
$K$ on $\Delta$ such that
$$\|\,x\,\|_*^{\,2}=\left\langle Kx,x\right\rangle_\Delta.$$
It is then easy to see that $\calH_*$ is exact interpolation if and only if, for bounded
operators $T$ on $\Delta$, the conditions $T^*HT\le H$ and $T^*(1-H)T\le 1-H$ imply
$T^*KT\le K$.
It is straightforward to check that the relations between $H$, $K$ and the operators $A$, $B$
used in the previous sections are:
\begin{equation}\label{rels}H=\frac 1 {1+A}\quad ,\quad A=\frac {1-H} H\quad ,\quad
K=\frac B {1+A}\quad ,\quad B=\frac K H.\end{equation}
(It follows from the proof of Lemma \ref{donogh} that $H$ and $K$ commute.)

By Theorem \ref{dthm} we know that $\calH_*$ is an exact interpolation space if and only if $B=h(A)$
for some $h\in P'$. By \eqref{rels}, this is equivalent to that $K=k(H)$ where
$$k(H)=\frac {h(A)} {1+A}=H\,h\left(\frac {1-H}H\right).$$
In its turn, this means that
\begin{align*}k(\lambda)&=\lambda\int_{[0,\infty]}\frac {(1+t)(1-\lambda)/\lambda}
{1+t(1-\lambda)/\lambda}\, d\vro(t)\\
&=\int_{[0,\infty]}\frac {(1+t)\lambda(1-\lambda)} {\lambda+t(1-\lambda)}\,d\vro(t),
\quad \lambda\in\sigma(H),
\end{align*}
where $\vro$ is a suitable Radon measure. Applying the change of variables $s=1/(1+t)$ and defining
a positive Radon measure $\mu$ on $[0,1]$ by $d\mu(s)=d\vro(t)$, we arrive at the expression
\begin{equation}k(\lambda)=\int_0^1\frac {\lambda(1-\lambda)}{(1-s)(1-\lambda)+s\lambda}
\, d\mu(s),\qquad \lambda\in \sigma(H),\end{equation}
which gives the representation exact quadratic interpolation spaces originally used by Donoghue in \cite{D1}.

\subsection{$J$-methods and the Foia\c{s}-Lions theorem} \label{jfl} We define the (quadratic) $J$-functional relative to a regular Hilbert couple $\overline{\calH}$ by
$$J(t,x)=J\left(t,x;\overline{\calH}\right)=\left\|\,x\,\right\|_0^{\, 2}+t\left\|\,x\,\right\|_1^{\,2},\qquad t>0,\, x\in \Delta(\,\overline{\calH}\,).$$
Note that $J(t,x)^{1/2}$ is an equivalent norm on $\Delta$ and that $J(1,x)=\left\|\,x\,\right\|_\Delta^{\,2}$.

Given a positive Radon measure $\nu$ on $[0,\infty]$, we define a Hilbert space $J_\nu(\,\overline{\calH}\,)$ as the set of all elements $x\in\Sigma(\,\overline{\calH}\,)$
such that there exists a measurable function $u:[0,\infty]\to \Delta$ such that
\begin{equation}\label{fu}x=\int_{[0,\infty]}u(t)\, d\nu(t)\quad\text{(convergence\, in\,}\Sigma\text{)}\end{equation}
and
\begin{equation}\label{fla}\int_{[0,\infty]}\frac  {J(t,u(t))} {1+t}\, d\nu(t)<\infty.\end{equation}
The norm in the space $J_\nu(\overline{\calH})$ is defined by
\begin{equation}\label{cho}\left\|\,x\,\right\|_{J_\nu}^{\,2}=\inf_u\int_{[0,\infty]}\frac {J(t,u(t))}{1+t}\, d\nu(t)\end{equation}
over all $u$ satisfying \eqref{fu} and \eqref{fla}.

The space \eqref{cho} was (with different notation) introduced by Foia\c{s} and Lions in the paper \cite{FL},
where it was shown that there is a unique minimizer $u(t)$ of the problem \eqref{cho}, namely
\begin{equation}u(t)=\fii_t(A)x\qquad \text{where}\qquad \fii_t(\lambda)=\frac{1+t}{1+t\lambda}\left(\int_{[0,\infty]}
\frac {1+s}{1+s\lambda}\,d\nu(s)\right)^{-1}.\end{equation}
Inserting this expression for $u$ into \eqref{cho}, one finds that
$$\left\|\,x\,\right\|_{J_\nu}^{\,2}=\left\langle h(A)x,x\right\rangle_0$$ where
\begin{equation}\label{foli}h(\lambda)^{-1}=\int_{[0,\infty]}\frac {1+t} {1+t\lambda}\, d\nu(t).
\end{equation}
It is easy to verify that the class of functions representable in the form \eqref{foli} for
some positive Radon measure $\nu$ coincides with
the class $P'$. We have thus arrived at the following result.

\begin{thm} \label{thm7.6} Every exact quadratic interpolation space $\calH_*$ can be represented isometrically in the form
$\calH_*=J_\nu(\overline{\calH})$
for some positive Radon measure $\nu$ on $[0,\infty]$. Conversely, any space of this form is an exact quadratic
interpolation space.
\end{thm}

In the original paper \cite{FL}, Foia\c{s} and Lions proved the less precise statement that each exact quadratic
interpolation method $F$ can be represented as $F=J_\nu$ for some positive Radon measure $\nu$.

\subsection{The relation between the $K$- and $J$-representations}
The assignment
$K_\vro=J_\nu$ gives rise to a non-trivial bijection $\vro\mapsto\nu$ of the set of
positive Radon measures on $[0,\infty]$. In this bijection, $\vro$ and $\nu$ are in correspondence if and only if
$$\int_{[0,\infty]}\frac {(1+t)\lambda}{1+t\lambda}\,d\vro(t)=\left(\int_{[0,\infty]}\frac {1+t} {1+t\lambda}\, d\nu(t)\right)^{-1}.$$
As an example, let us consider the geometric interpolation space (where $c_\theta=\pi/\sin(\pi\theta)$)
$$\left\|\,x\,\right\|_\theta^{\,2}=\left\langle A^{\,\theta} x,x\right\rangle_0=c_\theta\int_0^\infty t^{-\theta}K(t,x)\,\frac {dt} t.$$
The measure $\vro$ corresponding to this method is $d\vro_\theta(t)=\frac {c_\theta t^{-\theta}}{1+t}\, dt$.
On the other hand, it is easy to check that
$$\lambda^\theta=\left(\int_0^\infty \frac {1+t}{1+t\lambda}\, d\nu_\theta(t)\right)^{-1}\quad
\text{where}\quad d\nu_\theta(t)=\frac{c_\theta t^\theta}{1+t}\frac {dt} t.$$
We leave it to  the reader to check that
the norm in $\calH_\theta$ is the infimum of the expression
$$c_\theta\int_0^\infty {t^\theta}J(t,u(t))\,\frac {dt} t$$
over all
representations
$$x=\int_0^\infty
u(t)\,\frac {dt} t.$$
We have arrived at the Hilbert space version of Peetre's $J$-method of exponent $\theta$. The identity
$J_{\nu_\theta}=K_{\vro_\theta}$ can now be recognized as a sharp (isometric) Hilbert space version of
the equivalence theorem of Peetre, which says that the standard $K_\theta$ and $J_\theta$-methods give rise
to equivalent norms on the category of Banach couples (see \cite{BL}).

The problem of determining the pairs $\vro,\nu$ having the
property that the $K_\vro$ and $J_\nu$ methods give equivalent norms
was studied by Fan in \cite[Section 3]{F}.

\subsection{Other representations} As we have seen in the preceding subsections, using the space $\calH_0$ to express all involutions
and inner products leads to a description of the exact quadratic interpolation spaces
in terms of the class $P'$. If we instead use the space $\Delta$ as the basic object, we get Donoghue's representation for interpolation functions. Similarly, one can proceed from any fixed interpolation space
$\calH_*$ to obtain a different representation of interpolation functions.

\subsection{On interpolation methods of power $p$} \label{ointp} Fix a number $p$, $1<p<\infty$. We shall
write
$L_p=L_p\left(X,\calA,\mu\right)$ for the usual $L_p$-space associated with an arbitrary but fixed
($\sigma$-finite) measure $\mu$ on a measure
space $\left(X,\calA\right)$. Given a positive measurable weight function $w$, we write $L_p(w)$ for
the space normed by
$$\left\|\, f\,\right\|_{L_p(w)}^{\,p}=\int_X\left|f(x)\right|^{\,p}w(x)\, d\mu(x).$$
We shall write $\overline{L}_p(w)=\left(L_p,L_p(w)\right)$ for the corresponding weighted $L_p$ couple.
Note that the conditions imposed mean precisely that $\overline{L}_p(w)$ be separable and regular.

Let us say that an exact interpolation functor $F$ defined on the totality of separable, regular weighted
$L_p$-couples and taking values in the class of weighted $L_p$-spaces is \textit{of power $p$}.

Define, for a positive Radon measure $\vro$ on $[0,\infty]$,
an exact interpolation functor $F=K_\vro(p)$ by the definition
$$\left\|\, f\,\right\|_{F(\overline{L}_p(w))}^{\,p}:=\int_{[0,\infty]}
(1+t^{-\frac 1 {p-1}})^{\,p-1}K_p\left(t,f;\overline{L}_p(w)\right)\, d\vro(t).$$
We contend that $F$ is of power $p$.

Indeed, it is easy to verify that
$$K_p\left(t,f;\overline{L}_p(w)\right)=\int_X\left|f(x)\right|^{\,p}\frac {tw(x)} {(1+\left(tw(x)\right)^{\frac 1 {p-1}})^{\,p-1}}\, d\mu(x),$$
so Fubini's theorem gives that
$$\left\|\,f\,\right\|_{F(\overline{L}_p(w))}^{\,p}=\int_X\left|f(x)\right|^{\,p}h(w(x))\,
d\mu(x),$$
where
\begin{equation}\label{inp}h(\lambda)=\int_{[0,\infty]}\frac {(1+t^{\frac 1 {p-1}})^{\,p-1}\lambda}
{(1+\left(t\lambda\right)^{\frac 1 {p-1}})^{\,p-1}}\,\,d\vro(t),\qquad \lambda\in w(X).
\end{equation}
We have shown that $F(\overline{L}_p(w))=L_p(h(w))$, so $F$ is indeed of power $p$.

Let us denote by $\calK(p)$ the totality of positive functions $h$ on $\R_+$
representable in the form \eqref{inp} for some
positive Radon measure $\vro$ on $[0,\infty]$.

Further, let $\calI(p)$ denote
the class of all (exact) \textit{interpolation functions of power $p$}, i.e., those
positive functions $h$ on $\R_+$ having the property that
for each weighted $L_p$ couple $\overline{L}_p(w)$ and each bounded operator $T$
on $\overline{L}_p(w)$, it holds that $T$ is bounded on $L_p(h(w))$ and
$$\left\|\,T\,\right\|_{\,\bo\left(L_p(h(w))\right)}\le\left\|\,T\,\right\|_{\, \bo(\overline{L}_p(w))}.$$
The class $\calI(p)$ is in a sense the natural candidate for the class of "operator monotone functions
on $L_p$-spaces". The class $\calI(p)$ clearly forms a convex cone; it was shown by Peetre \cite{P1} that this
cone is contained in the
class of concave positive functions on $\R_+$ (with equality if $p=1$).

We have shown that $\calK(p)\subset\calI(p)$. By Theorem \ref{flthm}, we know that equality
holds when $p=2$. For other values of $p$ it does not seem to be known whether the class $\calK(p)$ exhausts
the class $\calI(p)$, but one can show that we would have $\calK(p)=\calI(p)$ provided that
each finite-dimensional $L_p$-couple $\overline{\ell_p^n}(\lambda)$ has the
$K_p$-property (or equivalently, the \Kpr, see \eqref{rem?}).
Naturally, the latter problem (about the $K_p$-property)
also seems to be open, but some comments on it are found in Remark \ref{remu}.

Let $\nu$ be a positive Radon measure on $[0,\infty]$.
In \cite{FL}, Foia\c{s} and Lions introduced a method, which we will denote by $F=J_\nu(p)$ in the following
way. Define the $J_p$-functional by
$$J_p\left(t,f;\overline{L}_p(\lambda)\right)=\left\|\,f\,\right\|_0^{\,p}+t\left\|\,f\,\right\|_1^{\,p},\quad f\in \Delta,\, t>0.$$
We then define an intermediate norm by
$$\left\|\,f\,\right\|_{F(\overline{L}_p(\lambda))}^{\,p}:=\inf \int_{[0,\infty]}
\left(1+t\right)^{\,-\frac 1 {p-1}}J_p\left(t,u(t);\overline{L}_p(\lambda)\right)\, d\nu(t),$$
where the infimum is taken over all representations
$$f=\int_{[0,\infty]} u(t)\, d\nu(t)$$
with convergence in $\Sigma$. It is straightforward to see that the method $F$ so defined
is exact; in \cite{FL} it is moreover shown that it is of power $p$. More precisely, it
is there proved that
$$\left\|\,f\,\right\|_{F(\overline{L}_p(\lambda))}^{\,p}=\int_X \left|f(x)\right|^{\,p}
h(w(x))\, d\mu(x),$$
where
\begin{equation}\label{folip} h(\lambda)^{\,-\frac 1 {p-1}}=\int_{[0,\infty]}
\frac {\left(1+t\right)^{\,\frac 1 {p-1}}} {\left(1+t\lambda\right)^{\frac 1 {p-1}}}\, d\nu(t),\qquad \lambda\in
w(X).
\end{equation}
Let us denote by $\calJ(p)$ the totality of functions $h$ representable in the form \eqref{folip}.
We thus have that $\calJ(p)\subset\calI(p)$. In view of our preceding remarks, we conclude that if all
weighted $L_p$-couples have the $K_p$ property, then necessarily $\calJ(p)\subset \calK(p)$.
Note that $\calJ(2)=\calK(2)$ by Theorem \ref{thm7.6}.

\section*{Appendix: The complex method is quadratic} \label{mccp}

Let $S=\{z\in\mathbf{C};\, 0\le\re z\le 1\}$. Fix a Hilbert couple
$\overline{\calH}$ and let $\calF$ be the set of functions $S\to\Sigma$ which are bounded
and continuous in $S$, analytic in the interior of $S$, and which maps the line
$j+i\mathbf{R}$ into $\calH_j$ for $j=0,1$. Fix $0<\theta<1$. The norm in the complex
interpolation space $C_\theta\left(\,\overline{\calH}\,\right)$ is defined by
$$\left\|\,x\,\right\|_{C_\theta\left(\,\overline{\calH}\,\right)}=\inf\left\{\left\|\,f\,\right\|_\calF;\, f(\theta)=x\right\}.\leqno{(*)}$$

Let $\calP$ denote the set of polynomials $f=\sum_1^N a_iz^i$ where $a_i\in\Delta$.
We endow $\calP$ with the inner product
$$\left\langle f,g\right\rangle_{M_\theta}=\sum_{j=0,1}\int_\R\left\langle f(j+it),g(j+it)\right\rangle_jP_j(\theta,t)\, dt,$$
where $\{P_0,P_1\}$ is the Poisson kernel for $S$,
$$P_j(\theta,t)=\frac {e^{-\pi t}\sin\theta\pi}{\sin^2\theta\pi+(\cos\theta\pi-(-1)^je^{-\pi t})^2}.$$
Let $M_\theta$ be the completion of $\calP$ with this inner product. It is easy to see that the
elements of $M_\theta$ are analytic in the interior of $S$, and that
evaluation map $f\mapsto f(\theta)$ is continuous on $M_\theta$. Let $N_\theta$ be the kernel
of this functional and define a Hilbert space $\calH_\theta$ by
$$\calH_\theta=M_\theta/N_\theta.$$
We denote the norm in $\calH_\theta$ by $\|\cdot\|_\theta$.

\begin{propA1}
$C_\theta\left(\,\overline{\calH}\,\right)=\calH_\theta$ with equality of norms.
\end{propA1}

\begin{proof} Let $f\in\calF$. By the Calderón lemma in \cite[Lemma 4.3.2]{BL}, we have the estimate
$$\log\left\|\,f(\theta)\,\right\|_{C_\theta(\,\overline{\calH}\,)}\le\sum_{j=0,1}\int_\R\log\|f(j+it)\|_jP_j(\theta,t)\, dt.$$
Applying Jensen's inequality, this gives that
$$\left\|\,f(\theta)\,\right\|_{C_\theta(\,\overline{\calH}\,)}\le (\sum_{j=0,1}\int_\R
\left\|\,f(j+it)\,\right\|_j^2P_j(\theta,t)\, dt)^{1/2}=\left\|\,f\,\right\|_{M_\theta}.$$
Hence $\calH_\theta\subset C_\theta(\,\overline{\calH}\,)$ and $\left\|\cdot\right\|_{C_\theta(\,\overline{\calH}\,)}\le
\|\cdot\|_\theta$. On the other hand, for $f\in\calP$ one has the estimates
$$\left\|\,f(\theta)\,\right\|_\theta\le \left\|\,f\,\right\|_{M_\theta}\le\sup\{\left\|\,f(j+it)\,\right\|_j;\, t\in \mathbf{R},\, j=0,1\}=
\left\|\,f\,\right\|_\calF,$$
whence $C_\theta(\,\overline{\calH}\,)\subset \calH_\theta$ and $\left\|\cdot\right\|_{C_\theta(\,\overline{\calH}\,)}\ge
\left\|\cdot\right\|_\theta$.
\end{proof}

It is well known that the method $C_\theta$ is of exponent $\theta$
(see e.g. \cite{BL}).
We have shown that $C_\theta$ is an exact quadratic interpolation method
of exponent $\theta$.

\subsection*{Complex interpolation with derivatives} In \cite[pp. 421-422]{F}, Fan considers the more general complex interpolation method $C_{\theta(n)}$
for the $n$:th derivative. This means that in (*), one consider representations $x=\frac 1 {n!}f^{(n)}(\theta)$ where $f\in\calF$; the complex
method $C_\theta$ is thus the special case $C_{\theta(0)}$.
It is shown in \cite{F} that, for $n\ge 1$, the $C_{\theta(n)}$-method is represented, up to equivalence of norms, by the quasi-power function
$h(\lambda)=\lambda^{\,\theta}/(1+\frac {\theta(1-\theta)}n\left|\,\log \lambda\,\right|)^{\,n}$.
The complex method with derivatives was introduced by Schechter \cite{S}; for more details
on that method, we refer to the list of references in \cite{F}.

.

\end{document}